\documentclass{amsart}
\usepackage{amsmath,amsthm,amsfonts,amssymb,amscd}
\usepackage[dvips]{graphicx}
\usepackage{psfrag}

\theoremstyle{plain}
\newtheorem{thm}{Theorem}[section]
\newtheorem{rk}[thm]{Remark}
\newtheorem{prop}[thm]{Proposition}
\newtheorem{clly}[thm]{Corollary}
\newtheorem{lemma}[thm]{Lemma}

\newtheorem{defi}[thm]{Definition}

\newtheorem{example}[thm]{Example}

\title[$2$-Riemannian manifolds]{$2$-Riemannian manifolds}
\author{C. Morales}
\address{Instituto de Matem\'atica\\ Universidade Federal do Rio de Janeiro, 
P. O. Box 68530, 21945-970 Rio de Janeiro, Brazil.} 
\email{morales@impa.br}
\author{M. Vilches}
\address{Departamento de An\'alise Matem\'atica\\
IME, Universidade do Estado do Rio de Janeiro, 20550-013,
Rio de Janeiro, Brazil}
\email{mavil@impa.br}
\thanks{2010 MSC: Primary 53C05, Secondary 46C50} 
\thanks{{\em Keywords and Phrases}: 
$2$-Riemannian metric, 2-inner Product, Stationary Vector field} 
\thanks{CM was partially supported by CNPq, FAPERJ and 
PRONEX/DYN-SYS. from Brazil.}

\begin{document}
\maketitle
\begin{abstract}
A {\em $2$-Riemannian manifold} is a differentiable manifold exhibiting a $2$-inner product on each tangent space.
We first study lower dimensional $2$-Riemannian manifolds by giving necessary and sufficient conditions for flatness. Afterward we 
associate to each $2$-Riemannian manifold a unique torsion free compatible pseudoconnection. Using it we define a curvature for $2$-Riemannian manifolds and study its properties.
We also prove
that $2$-Riemannian pseudoconnections do not have Koszul derivatives.
Moreover, we define stationary vector field with respect to a $2$-Riemannian metric
and prove that the stationary vector fields in $\mathbb{R}^2$ with respect to the $2$-Riemannian metric induced by the Euclidean product are the divergence free ones.
\end{abstract}


\section{Introduction}

\noindent
In his famous 1854's {Habilitationsvortrag} {\em "Uber die Hypothesen, welche der Geometrie zu Grunde liegen"} Bernhard Riemann gave the foundations of the {\em Riemannian geometry} 
based on the choice of an inner product on each tangent space.
Afterward Finsler in his thesis \cite{f} developed a geometry,
the {\em Finsler geometry},
in which a general norm is chosen instead.
In 1933 Cartan \cite{ca} considered a geometry
based on the notion of area, the {\em Cartan spaces},
which are the dual of the Finsler spaces under the Legendre transformation \cite{m}.
Around 1950 Kawaguchi generalized both Finsler and Cartan
by introducing the {\em Areal Spaces} where
the $m$-dimensional area is given by a fundamental integral \cite{ka}.
More recently, Miron defined the {\em Hamilton spaces}
as a natural generalization of the Cartan spaces \cite{m2}.

In 1964 Gaehler introduced the concept of $2$-norm as
an abstraction of the area of the parallelogram formed by two vectors in a vector space. $2$-normed counterpart of well known results in the theory of normed spaces have been obtained
elsewhere \cite{g}.
Kawaguchi pointed out the equivalence between
a $2$-norm space and what he called flat areal space, i.e.,
a linear areal
space where the $2$-dimensional area does not depend on the base point
(see \cite{k2} p. 166).
In 1973 Diminnie, Gaehler and White \cite{dgw}
introduced the $2$-inner product spaces
which are $2$-dimensional analogue of inner product spaces.
The corresponding theory studying the relation between
inner, $2$-inner and $2$-normed spaces have been developed
elsewhere (e.g. \cite{clk}).

These works motivate the study of
{\em $2$-Riemannian manifolds}, that is,
differentiable manifolds
exhibiting a $2$-inner product on each tangent space
($n$-Riemannian manifolds may be defined analogously
using $n$-inner products \cite{mi} instead).

We first study $2$-Riemannian manifolds of dimension $2$ and $3$
by giving necessary and sufficient conditions
for locally flatness.
Afterward we 
associate to each $2$-Riemannian manifold a unique torsion free compatible pseudoconnection. Using it we
define a curvature for $2$-Riemannian manifolds and study its properties. We prove
that a $2$-Riemannian pseudoconnection does not have Koszul derivatives in the sense of \cite{k}.
We also define stationary vector field with respect to a $2$-Riemannian metric
based on Definition 3.1.3 in \cite{k}.
It is proved that the stationary vector fields in $\mathbb{R}^2$ with respect to the $2$-Riemannian metric induced by the standard Euclidean metric are precisely the divergence free ones.

\section*{Acknowledgements}

\noindent
The authors are grateful to professors Ch. Diminnie and
S. Hayashi for sending us the references \cite{dgw}, \cite{dgw2} and \cite{k2} respectively.

\section{$2$-Riemannian metrics}

\noindent
In this section we define $2$-Riemannian metrics on a differentiable manifold.
Previously we shall recall the definition of $2$-inner
product spaces \cite{dgw}.

\begin{defi}
\label{2-inner}
A {\em $2$-inner product} in a vector space $V$ is a map $g: V\times V\times V\to \mathbb{R}$ which
satisfies the following properties for all $u,u',v,w\in V$ and $\alpha\in \mathbb{R}$:
\begin{enumerate}
 \item
$g(u,u/v)\geq 0$ and $g(u,u/v)=0$ if and only if $u$ and $v$ are linearly dependent;
\item
$g(u,u/v)=g(v,v/u)$;
\item
$g(u,v/w)=g(v,u/w)$;
\item
$g(\alpha u+u',v/w)=\alpha g(u,v/w)+g(u',v/w)$.
\end{enumerate}
\end{defi}

{\em (Here we have used the customary notation $g(u,v/w)$ instead
of $g(u,v,w)$}.)

Let us quote some basic properties of $2$-inner products
on vector spaces (see for instance \cite{dgw} and Theorem 2 p. 271 in \cite{dgw2}).

\begin{lemma}
 \label{sl-dgw}
Let $g$ be a $2$-inner product in a vector space $V$. Then
\begin{enumerate}
\item
$g(e_1,e_2/e_1)=0$, $g(e_1,e_1/e_2)=-g(e_1,e_2/e_1+e_2)$, $g(e_1,e_2/\alpha_3e_3)=\alpha_3^2g(e_1,e_2/e_3)$;
\item
$
g\left( \sum_{i=1}^2\alpha_ie_i,\sum_{i=1}^2\beta_ie_i\bigg/\sum_{i=1}^2\gamma_ie_i\right)=
\det
\left(
\begin{matrix}
\alpha_1 &\gamma_1\\
\alpha_2 & \gamma_2
\end{matrix}
\right)
\det
\left(
\begin{matrix}
\beta_1 &\gamma_1\\
\beta_2 &\gamma_2
\end{matrix}
\right)
\cdot$
$$
g(e_1,e_1/e_2);
$$
\item
$
g\left(\sum_{i=1}^3\alpha_ie_i,\sum_{i=1}^3\beta_ie_i\bigg/\sum_{i=1}^3\gamma_ie_i\right)
=\frac{1}{2}\sum_{i,j=1,i\neq j}^3
\det
\left(
\begin{matrix}
\alpha_i &\gamma_i\\
\alpha_j & \gamma_j
\end{matrix}
\right)
\cdot
$
$$
\det
\left(
\begin{matrix}
\beta_i &\gamma_i\\
\beta_j & \gamma_j
\end{matrix}
\right)\cdot
g(e_i,e_i/e_j)+
\sum_{i,j,k=1,i\neq j\neq k\neq i}^3
\det
\left(
\begin{matrix}
\alpha_i &\gamma_i\\
\alpha_k & \gamma_k
\end{matrix}
\right)
\det
\left(
\begin{matrix}
\beta_j &\gamma_j\\
\beta_k & \gamma_k
\end{matrix}
\right)\cdot
$$
$$
g(e_i,e_j/e_k),
$$
\end{enumerate}
for all $e_1,e_2,e_3\in V$, $\alpha_i,\beta_i,\gamma_i\in \mathbb{R}$ ($i=1,2,3$).
\end{lemma}

\subsection{Definition}
Let $M$ be a differentiable manifold. Denote by
$C^\infty(M)$ the ring of all $C^\infty$ real-valued
functions in $M$.
Let $\Pi:\xi\to M$ be a vector bundle over $M$ with total space $\xi$ and projection $\Pi$
(we write $\xi$ instead of $\Pi:\xi\to M$ for simplicity).
Denote by $\Omega^0(\xi)$ the
$C^\infty(M)$-module of $C^\infty$ cross-sections of $\xi$.
Whenever $\xi=TM$ is the tangent bundle of $M$ we shall write
$\mathcal{X}$ instead of $\Omega^0(TM)$ (or $\mathcal{X}(M)$ to emphasize dependence on $M$).
The fiber of $\xi$ over $p$ is denoted by $\xi_p$.

\vspace{5pt}

A {\em $2$-Riemannian metric} in $\xi$ is
a map $g$ assigning a $2$-inner product $g_p$ to each fiber $\xi_p$, $p\in M$, which is smooth in the following sense:
For every $r,s,t\in \Omega^0(\xi)$
the map $g( r,s/t): M\to \mathbb{R}$ defined by
$$
g(r,s/t)(p)=
g_p(r(p),s(p)/t(p))
$$
belongs to $C^\infty(M)$.
A {\em $2$-Riemannian metric} in $M$ is a $2$-Riemannian metric in its
tangent bundle. A {\em $2$-Riemannian manifold} is a pair $(M,g)$
where $M$ is a manifold and $g$ is a $2$-Riemannian metric in $M$.

\vspace{5pt}

Every $2$-Riemannian metric $g$ in $M$ induces an areal metric \cite{ka} defined by
$$
F(u,v)=\sqrt{g_p( u,u/v)},
\,\,\,\,\,\,\forall p\in M,\forall u,v\in T_pM.
$$
Another related (although different) concept is that of area metric manifold in \cite{psw}.
Indeed, a $2$-Riemannian metric cannot be represented by a tensor
field as in the area metric case.

Every Riemannian metric $h$ in $M$ induces
a $2$-Riemannian metric in $M$ defined by
\begin{equation}
\label{simple}
g_p(u,v/w\rangle=h_p( u,v)\cdot h_p(w,w)-h_p( u,w)\cdot h_p( v,w),
\end{equation}
for all $p\in M$ and $u,v,w\in T_pM$.
This is called the {\em simple} $2$-Riemannian metric generated by $h$. A $2$-Riemannian metric $g$ in $M$ is
simple if it is the simple $2$-Riemannian metric generated by
some Riemannian metric in $M$.

\subsection{The locally flatness problem}
The basic problem in Riemannian geometry is to give necessary and sufficient conditions for a Riemannian manifold to be locally flat.
Such a problem gave rise to the concept of curvature in Riemannian
geometry (see for instance \cite{s}).
In this subsection we want to formulate the analogous problem but
for $2$-Riemannian metrics instead.
For this we use the following definition.

\begin{defi}
A {\em $2$-isometry} between $2$-Riemannian manifolds
$(M,g)$ and $(\overline{M},\overline{g})$ is a diffeomorphism
$h:M\to \overline{M}$ satisfying
$$
\overline{g}_{h(p)}(Dh_p(u),Dh_p(v)/Dh_p(w))=g_p(u,v/w),
\,\,\,\,\forall p\in M, \forall u,v,w\in T_pM.
$$
We say that $(M,g)$ and $(\overline{M},\overline{g})$ are
{\em $2$-isometric} if there is a $2$-isometry between them; and
{\em locally $2$-isometric} if
for every $p\in M$ there are a neighborhood $U$ of $p$ and a neighborhood $\overline{U}$ in $\overline{M}$ such that
$(U,g)$ and $(\overline{U},\overline{g})$ are $2$-isometric.
We say that $M$ is {\em locally flat} if
it is locally $2$-isometric to $(\mathbb{R},g^{st})$, where
$n=dim(M)$ and
$g^{st}$ is
the simple $2$-Riemannian metric of $\mathbb{R}^n$ induced by the standard Euclidean product of $\mathbb{R}^n$.
\end{defi}

The locally flatness problem then consists of giving necessary and sufficient conditions for a $2$-Riemannian manifold $(M,g)$ to be locally flat.
This problem makes sense only in dimension $\geq 3$ by the following result.

\begin{thm}
\label{flat}
Every $2$-Riemannian manifold of dimension $2$ is locally flat.
\end{thm}

\begin{proof}
Let $M$ be a Riemannian manifold of dimension $2$. Fix $p_0\in M$ and
denote by $(x,y)$ the standard coordinate system of $\mathbb{R}^2$.
Choose a coordinate system $(U,\phi)$ around $p_0$.
Define $G:\phi(U)\to \mathbb{R}$ by
$$
G(q)=g_{\phi(q)}\left(D\phi_q\left(\frac{\partial}{\partial x}\right),
D\phi_q\left(\frac{\partial}{\partial x}\right)\bigg/D\phi_q\left(\frac{\partial}{\partial y}
\right)\right).
$$

Let $(u,v)=(u(x,y),v(x,y))$ be a solution of the PDE below:
$$
\frac{\partial(u,v)}{\partial(x,y)}=\sqrt{G},
$$
where
$$
\frac{\partial(u,v)}{\partial(x,y)}=\frac{\partial u}{\partial x}\frac{\partial v}{\partial y}-\frac{\partial u}{\partial y}\frac{\partial v}{\partial x}
$$
stands for the Jacobian of $(u,v)$ with respect to $(x,y)$.
(Such a solution always exists e.g. \cite{kn}).

Define $h:U\to \mathbb{R}^2$ by
$h=(u,v)\circ \phi^{-1}$. Clearly $h$ is a diffeomorphism onto
$V=h(U)$.
Take $p\in \phi(U)$ and $q=\phi^{-1}(p)$.
The chain rule yields
$$
Dh_{\phi(q)}\left(D\phi_q\left(\frac{\partial}{\partial z}\right)\right)=
D(u,v)_q\left(\frac{\partial}{\partial z}\right)
$$
for $z=x,y$.
By Lemma \ref{sl-dgw}-(2) we obtain
$$
g^{st}_{h(p)}\left(Dh_{p}\left(D\phi_q\left(\frac{\partial}{\partial x}\right)\right),Dh_{p}\left(D\phi_q\left(\frac{\partial}{\partial x}\right)\right)\bigg/Dh_{p}\left(D\phi_q\left(\frac{\partial}{\partial y}\right)\right)\right)=
$$
$$
\left(\frac{\partial(u,v)}{\partial(x,y)}\right)^2=G.
$$
Thus the definition of $G$ above implies
\begin{equation}
\label{2-isometry}
g^{st}_{h(p)}(Dh_{p}(a),Dh_{p}(a)/Dh_{p}(b))
=
g_{p}(a,a/b)
\end{equation}
at least for
$a=D\phi_q\left(\frac{\partial}{\partial x}\right)$ and
$b=D\phi_q\left(\frac{\partial}{\partial y}\right)$.
It follows that (\ref{2-isometry}) holds for all $a,b\in T_pM$ since
$\{D\phi_q\left(\frac{\partial}{\partial x}\right),
D\phi_q\left(\frac{\partial}{\partial y}\right)\}$ is a base of $T_pM$.
As $p\in \phi(U)$ is arbitrary we conclude that $h$ is a $2$-isometry
between $(\phi(U), g)$ and $(V, g^{st})$.
Since $p_0\in M$ is arbitrary we conclude that
$M$ is locally flat.
This proves the result.
\end{proof}

Theorem \ref{flat} is a $2$-dimensional version
of an elementary fact in Riemannian geometry (\cite{lee} p. 116) asserting that
every Riemannian $1$-manifold is locally isometric to $\mathbb{R}$ (or even that
every curve can be parametrized by arc length).
It is also related to the well known existence of isothermal
coordinates on every Riemannian surface \cite{s}.
It can be used as well to prove that
every $2$-Riemannian manifold of dimension $2$ is simple and,
furthermore, that two arbitrary $2$-Riemannian manifolds of dimension $2$ are locally $2$-isometric.

Now we consider the {\em $2$-Riemannian $3$-manifolds}, i.e.,
$2$-Riemannian manifolds of dimension $3$.
In such a case we have the following characterization.

\begin{thm}
\label{thh1}
A $2$-Riemannian $3$-manifold $(M,g)$ is locally flat if and only if
for every $p_0\in M$ there is a coordinate system $(O,\phi)$ around
$p_0$ such that the following system of first order PDE,
$$
\sum_{1\leq \alpha<\beta\leq 3}
\frac{\partial(f^\alpha,f^\beta)}{\partial(x^i,x^k)}\cdot\frac{\partial(f^\alpha,f^\beta)}{\partial(x^j,x^k)}=g_{ijk},
\,\,\,\,i,j,k\in \{1,2,3\}
$$
has a solution $f=(f^1,f^2,f^3):O\to \mathbb{R}^3$ where
$(x^1,x^2,x^3)$ is the standard coordinate system of $\mathbb{R}^3$ and the
$g_{ijk}$'s are the functions defined by
$$
g_{ijk}(q)=g_{\phi(q)}\left(D\phi_q\left(\frac{\partial}{\partial x^i}\right),
D\phi_q\left(\frac{\partial}{\partial x^j}\right)\bigg/D\phi_q\left(\frac{\partial}{\partial x^k}\right)\right).
$$
\end{thm}

\begin{proof}
Let us prove the sufficiency.
Consider a locally flat $2$-Riemannian $3$-manifold $(M,g)$ and
pick $p_0\in M$. Then, there is a diffeomorphism
$h:O_0\to O$ from a neighborhood $O_0$ of $p_0$
onto an open set $O\subset \mathbb{R}^3$ such that
\begin{equation}
\label{eqq1}
g^{st}_{h(p)}(Dh_p(u),Dh_p(v)/Dh_p(w))=g_p(u,v/w),
\quad \forall p\in O_0,\forall u,v,w\in T_pM.
\end{equation}

We define $\phi=h^{-1}:O\to O_0$ hence $(O,\phi)$ is a coordinate system
around $p_0$. For all $i,j,k\in \{1,2,3\}$ and all $q\in O$ we have
\begin{equation}
\label{eqq2}
g_{ijk}(q)=
g^{st}_q\left(\frac{\partial}{\partial x^i},\frac{\partial}{\partial x^j}
\bigg/\frac{\partial}{\partial x^k}\right)=\delta_{ijk},
\end{equation}
where the symbol $\delta_{ijk}$ is defined by
$$
\delta_{ijk}=\delta_{ij}-\delta_{ik}\cdot\delta_{jk}.
$$
($\delta_{ij}$ is the Kronecker delta.)
It follows from (\ref{eqq1}) and (\ref{eqq2}) that the system in the statement has the trivial
solution $f=(x^1,x^2,x^3)$.
This proves the only if part of the theorem.

Now we prove the necessity.
Suppose that there is a coordinate system
$(O,\phi)$ around a fixed (but arbitrary) point
$p_0\in M$ so that the aforementioned system has a solution
$f=(f^1,f^2,f^3)$.
It turns out that $f$ is a diffeomorphim onto its image $f(O)$ which is an open subset of $\mathbb{R}^3$.

Define
$h: \phi(O)\to f(O)$ by $h=f\circ \phi^{-1}$.
Noting that
$$
Dh_p\left(D\phi\left(\frac{\partial}{\partial x^i}\right)\right)=
\sum_\alpha \frac{\partial f^\alpha}{\partial x^i} \cdot
\frac{\partial}{\partial x^\alpha}
$$
we get
$$
g^{st}_{h(p)}
\left(Dh\left(D\phi\left(\frac{\partial}{\partial x^i}\right)\right),
Dh\left(D\phi\left(\frac{\partial}{\partial x^j}\right)\right)\bigg/
Dh\left(D\phi\left(\frac{\partial}{\partial x^k}\right)\right)\right)=
$$
$$
g_{ijk}(\phi^{-1}(p))
$$
because of the system in the statement
(we have written $Dh$ instead of $Dh_p$, etc for simplicity).
Hence
$$
g^{st}_{h(p)}
\left(Dh\left(D\phi\left(\frac{\partial}{\partial x^i}\right)\right),
Dh\left(D\phi\left(\frac{\partial}{\partial x^j}\right)\right)\bigg/
Dh\left(D\phi\left(\frac{\partial}{\partial x^k}\right)\right)\right)=
$$
$$
g_p\left(D\phi\left(\frac{\partial}{\partial x^i}\right),D\phi\left(\frac{\partial}{\partial x^j}\right)\bigg/
D\phi\left(\frac{\partial}{\partial x^k}\right)\right)
$$
by the definition of $g_{ijk}$.

Now pick $u,v,w\in T_pM$. Then,
$$
u=\sum\alpha_i D\phi\left(\frac{\partial}{\partial x^i}\right),
\,\,\,\,v=
\sum \beta_iD\phi\left(\frac{\partial}{\partial x^i}\right)
,
\,\,\,\,
w=\sum\gamma_iD\phi\left(\frac{\partial}{\partial x^i}\right)
$$
for some scalars $\alpha_i,\beta_i,\gamma_i$.
Applying Lemma \ref{sl-dgw}-(3) we obtain
$$
g_p(u,v/w)=
\sum_{1\leq i<j\leq3}
\det
\left(
\begin{matrix}
\alpha_i &\gamma_i\\
\alpha_j & \gamma_j
\end{matrix}
\right)
\det
\left(
\begin{matrix}
\beta_i &\gamma_i\\
\beta_j & \gamma_j
\end{matrix}
\right)
\cdot
$$
$$
g^{st}_{h(p)}\left(Dh\left(D\phi\left(\frac{\partial}{\partial x^i}\right)\right),
Dh\left(D\phi\left(\frac{\partial}{\partial x^i}\right)\right)\bigg/
Dh\left(D\phi\left(\frac{\partial}{\partial x^j}\right)\right)\right) +
$$
$$
\sum_{i\neq j\neq k\neq i}
\det
\left(
\begin{matrix}
\alpha_i &\gamma_i\\
\alpha_k & \gamma_k
\end{matrix}
\right)
\det
\left(
\begin{matrix}
\beta_j &\gamma_j\\
\beta_k & \gamma_k
\end{matrix}
\right)
\cdot
$$
$$
g^{st}_{h(p)}\left(Dh\left(D\phi\left(\frac{\partial}{\partial x^i}\right)\right),
Dh\left(D\phi\left(\frac{\partial}{\partial x^j}\right)\right)\bigg/
Dh\left(D\phi\left(\frac{\partial}{\partial x^k}\right)\right)\right)=
$$
$$
\sum_{1\leq i<j\leq3}
\det
\left(
\begin{matrix}
\alpha_i &\gamma_i\\
\alpha_j & \gamma_j
\end{matrix}
\right)
\det
\left(
\begin{matrix}
\beta_i &\gamma_i\\
\beta_j & \gamma_j
\end{matrix}
\right)
\cdot
$$
$$
g_p\left(D\phi\left(\frac{\partial}{\partial x^i}\right),D\phi\left(\frac{\partial}{\partial x^i}\right)\bigg/
D\phi\left(\frac{\partial}{\partial x^j}\right)\right) +
$$
$$
\sum_{i\neq j\neq k\neq i}
\det
\left(
\begin{matrix}
\alpha_i &\gamma_i\\
\alpha_k & \gamma_k
\end{matrix}
\right)
\det
\left(
\begin{matrix}
\beta_j &\gamma_j\\
\beta_k & \gamma_k
\end{matrix}
\right)
\cdot
$$
$$
g_p\left(D\phi\left(\frac{\partial}{\partial x^i}\right),D\phi\left(\frac{\partial}{\partial x^j}\right)\bigg/
D\phi\left(\frac{\partial}{\partial x^k}\right)\right)=
$$
$$
g^{st}_{h(p)}(Dh(u),Dh(v)/Dh(w))
$$
and the result follows.
\end{proof}

We shall apply this result to find non-locally flat $2$-Riemannian metrics in
$\mathbb{R}^3$. For this we need the following lemmas.
Let $I$ be the identity matrix.
Given a differentiable map $f$ we denote by $D^tf(x)$ the transposed of the matrix $Df(x)$.

\begin{lemma}
\label{propp1}
Let $\mu$ be a positive non-constant $C^\infty$ function defined in an open subset $\Omega\subset \mathbb{R}^3$.
If the following PDE
$$
D^tf(x)\cdot Df(x)= J(x,f)^2\cdot \mu(x)\cdot I
$$
has a $C^\infty$ solution $f$ with non-zero Jacobian $J(x,f)$,
then
$$
\mu(x)=\left(\frac{\mid x-a\mid^4}{r^4}\right)^2,
$$ for some $r>0$ and some $a\notin \Omega$.
\end{lemma}

\begin{proof}
The proof uses standard arguments in Geometric Function Theory \cite{im}.

Observe that if $J(x,f)^2\mu(x)$ were constant, then $J(x,f)$ also does
(e.g. replace in the PDE of the statement and take determinant).
So, $\mu(x)$ would be constant which is a contradiction.
Therefore $J(x,f)^2\mu(x)$ is not a constant function.

Next we define
$$
\lambda(x)=\ln(J(x,f)^2\mu(x)).
$$
The hypothesis of the proposition implies that there is a solution of
$$
D^tf(x)\cdot Df(x)=e^{\lambda(x)}I.
$$
From this we get (as in \cite{im} p. 39) that $\lambda$ solves
$$
4\frac{\partial^2\lambda}{\partial x^i\partial x^j}=
2\frac{\partial \lambda}{\partial x^i}\cdot\frac{\partial \lambda}{\partial x^j}
\,\,\,\,\,
-\,\,\,\,\mid \nabla\lambda\mid^2\cdot \delta_{ij},
$$
where $\nabla$ is the gradient operator and $\mid \cdot\mid$ is the norm operation.
Set
$$
P(x)=e^{-\frac{\lambda(x)}{2}}.
$$
It turns out that
$\mid \nabla P\mid^2=\frac{P^2}{4}\mid\nabla \lambda\mid^2$ so
$
\frac{\partial^2 P}{\partial x^i\partial x^j}=
\frac{\mid \nabla P\mid ^2}{2P}\delta_{ij}
$
thus $P$ solves
$$
2P\frac{\partial^2P}{\partial x^i\partial x^j}=
\mid \nabla P\mid^2\delta_{ij}.
$$
Because of this we get
$$
P(x)=r^{-2}\mid x-a\mid^2
$$
for some $r>0$ and some $a\notin \Omega$
(see for instance \cite{im} p. 40).
Then,
$$
e^{-\frac{1}{2}\ln(J(x,f)^2\mu(x))}=
r^{-2}\mid x-a\mid^2
$$
and so
$J(x,f)^{-1}\mu(x)^{-\frac{1}{2}}=r^{-2}\mid x-a\mid^2$,
yielding
\begin{equation}
\label{im1}
J(x,f)^2\mu(x)=\frac{r^4}{\mid x-a\mid^4}.
\end{equation}
Replacing in the PDE of the statement we get
$$
D^tf(x)\cdot Df(x)=\frac{r^4}{\mid x-a\mid^4} I.
$$
By taking determinant in both sides of the above equation we obtain
$$
J(x,f)^2=\left(\frac{r^4}{\mid x-a\mid^4}\right)^3
$$
and replacing in (\ref{im1}) we obtain
$$
\left(\frac{r^4}{\mid x-a\mid^4}\right)^3\mu(x)=\frac{r^4}{\mid x-a\mid^4}
$$
which proves the result.
\end{proof}

\begin{lemma}
\label{propp2}
Let $g_{ijk}$ be $C^\infty$ real valued functions defined in an open subset
$\Omega$ of $\mathbb{R}^3$, $i,j,k\in \{1,2,3\}$.
Suppose that the matrix
$$
G(x)=\left(\begin{matrix}
g_{223}& -g_{123}& -g_{132}\\
-g_{123}& g_{113}& -g_{231}\\
-g_{132}& -g_{231} & g_{112}
\end{matrix}
\right)^{-1}
$$
is well defined for all $x\in \Omega$.
A $C^\infty$ diffeomorphism $f=(f^1,f^2,f^3)$ defined in $\Omega$ is a solution of the PDE
$$
g_{ijk}=
\sum_{1\leq \alpha<\beta\leq 3}
\frac{\partial(f^\alpha,f^\beta)}{\partial(x^i,x^k)}\cdot\frac{\partial(f^\alpha,f^\beta)}{\partial(x^j,x^k)}
$$
if and only if $f$ solves the
following Beltrami-like system
$$
D^tf(x)\cdot Df(x)=J(x,f)^2\cdot G(x).
$$
\end{lemma}

\begin{proof}
Let $Cof(A)$ be the cofactor matrix of $A$ (e.g. \cite{c}).
The hypothesis on $f$ implies that
$$
Cof(D^tf)\cdot Cof(Df)=G^{-1}.
$$
But $Cof(A)=det(A)\cdot (A^{-1})^t$
wherever $A$ is an invertible matrix.
Taking $A=Df$ we get
$$
Cof(Df)=J(\cdot,f)\cdot((Df)^{-1})^t.
$$
(See \cite{c} p. 4.)
Hence $Cof(Df)^t=J(\cdot, f)\cdot(Df)^{-1}$ and
so
$$
J(x,f)^2\cdot(Df(x))^{-1}\cdot((Df)^{-1})^t=G^{-1}(x)
$$
which is equivalent to the Beltrami-like system in the statement.
\end{proof}

\begin{thm}
\label{thh2}
There is a $2$-Riemannian metric in $\mathbb{R}^3$ which is
not locally flat.
\end{thm}

\begin{proof}
Just take a $2$-Riemannian metric {\em conformally equivalent to $g^{st}$}, i.e.,
$g=\lambda\cdot g^{st}$ for some positive function $\lambda$.
If $g$ were locally flat then we would have by Theorem \ref{thh1} and Lemma \ref{propp2} that
the PDE in Lemma \ref{propp1} has a solution $f$ for that $\lambda$.
Hence $\mu$ would be as in the conclusion of Lemma \ref{propp1}.
The result then follows by taking a suitable $\lambda$.
\end{proof}

It can be proved however that, as in the surface case,
every $2$-Riemannian $3$-manifold is simple.

\section{$\mathfrak{D}$-Pseudoconnections}

\noindent
In this section we define the kind of connection
which we will associate to a $2$-Riemannian manifold.

\subsection{Definition}
Recall that an ordinary (or affine) connection of a vector bundle $\xi$ over a manifold $M$ is defined in classical differential geometry as an $\mathbb{R}$-bilinear map $\nabla:\mathcal{X}\times \Omega^0(\xi)\to\Omega^0(\xi)$
satisfying $\nabla_{\varphi X}s=\varphi\nabla_Xs$ and $\nabla_X(\varphi s)=X(\varphi) s+\varphi \nabla_Xs$
for all $(X,s)\in \mathcal{X}\times \Omega^0(\xi)$ and all
$\varphi\in C^\infty(M)$.
This definition was extended in
\cite{a2} where the last property is replaced by
$\nabla_X(\varphi s)=X(\varphi) P(s)+ \varphi\nabla_Xs$ for some
homomorphism $P:\Omega^0(\xi)\to \Omega^0(\xi)$.
A further extension \cite{a1} allows $P$ to take values not necessarily in $\Omega^0(\xi)$ but in
$\Omega^0(\eta)$ for some another vector bundle $\eta$
over $M$.
Here we have the necessity of further extending
the definition by allowing
$P$ to take values in an arbitrary
module $\mathfrak{D}$ over $C^\infty(M)$.
The precise definition is as follows.

\vspace{5pt}

A pseudoconnection with values in $\mathfrak{D}$ of $\xi$
(or {\em $\mathfrak{D}$-pseudoconnection} for short)
is an $\mathbb{R}$-bilinear map
$\nabla:\mathcal{X}\times \Omega^0(\xi)\to \mathfrak{D}$ for which
there is
a homomorphism $P:\Omega^0(\xi)\to \mathfrak{D}$, called the
{\em principal homomorphism} of $\nabla$,
such that
$\nabla_{\varphi X}s=\varphi\nabla_Xs$
and $\nabla_X(\varphi s)=X(\varphi) P(s)+\varphi\nabla_Xs$
for all $(X,s)\in \mathcal{X}\times \Omega^0(\xi)$ and $\varphi\in C^\infty(M)$.
(We shall use the customary notation $\nabla_XY$ instead of $\nabla(X,Y)$.)
When $\xi=TM$ is the tangent bundle of $M$
we say that $\nabla$ is {\em torsion free} if $
\nabla_XY-\nabla_YX=P([X,Y])$ for all $X,Y\in \mathcal{X}$.

\vspace{5pt}

The class of $\mathfrak{D}$-pseudoconnections
is broad enough to include not only the {\em ordinary connections} (where $\mathfrak{D}=\Omega^0(\xi)$ and $P$ is the identity) and
the {\em ordinary pseudoconnections} (where $\mathfrak{D}=\Omega^0(\xi)$)
but also the {\em $O$-derivative operators} where
$\mathfrak{D}=\Omega^0(\eta)$ for some vector bundle $\eta$ over
$M$.

Let us introduce a basic example of $\mathfrak{D}$-pseudoconnection
on a vector bundle $\xi$ over $M$. Given a $C^\infty(M)$-module $\mathfrak{D}$
we denote by $\Omega^1(\xi,\mathfrak{D})$ the set
of all homomorphisms $\omega:\Omega^0(\xi)\to \mathfrak{D}$.
This is a $C^\infty(M)$-module under the usual operations.
For simplicity we write
$\Omega^1(M,\mathfrak{D})$ instead of
$\Omega^1(TM,\mathfrak{D})$.

For every $r\geq 0$
the module $\mathfrak{D}$ acts on the right on the
$C^\infty(M)$-module of all maps
$\varphi:\mathcal{X}\times\overset{(r)}{\cdots}\times \mathcal{X}\times \Omega^0(\xi)\to C^\infty(M)$. This action is defined by
$$
(\varphi\cdot d)_{(X_1,\cdots,X_r)}s=\varphi(X_1,\cdots, X_r,s)d.
$$
In particular for $r=0$ we
have $(\varphi\cdot d)s=\varphi(s) d$, $\forall s\in \Omega^0(\xi)$
$\forall \varphi:\Omega^0(\xi)\to C^\infty(M)$.
It is clear that $\varphi \cdot d\in \Omega^1(\xi,\mathfrak{D})$ whenever
$(\varphi,d)\in\Omega^1(\xi,C^\infty(M))\times \mathfrak{D}$.

On the other hand, we define the differential
$\partial\varphi:\mathcal{X}\times \Omega^0(\xi)\to C^\infty(M)$
of any map $\varphi:\Omega^0(\xi)\to C^\infty(M)$ by
\begin{equation}
\label{differential}
\partial\varphi_Xs=X(\varphi(s)),
\,\,\,\,\,\,\forall (X,s)\in\mathcal{X}\times \Omega^0(\xi).
\end{equation}
It is a $C^\infty(M)$-pseudoconnection of $\xi$ with principal homomorphism $\varphi$ whenever $\varphi\in\Omega^1(\xi,C^\infty(M))$.
Therefore,
$\partial\varphi \cdot d$ is a $\mathfrak{D}$-pseudoconnection
of $\xi$ with principal homomorphism $\varphi\cdot d$,
$\forall (\varphi,d)\in\Omega^1(\xi,C^\infty(M))\times \mathfrak{D}$.

We also define the product $\varphi \cdot\omega:\mathcal{X}\times\Omega^0(\xi)\to\mathfrak{D}$
between $\varphi:\Omega^0(\xi)\to C^\infty(M)$ and $\omega:\mathcal{X}\to \mathfrak{D}$ by
$$
(\varphi\cdot\omega)_Xs=\varphi(s)\omega(X),
\,\,\,\,\,\forall (X,s)\in \mathcal{X}\times\Omega^0(\xi).
$$
Then, $\varphi\cdot\omega$ is
a $\mathfrak{D}$-pseudoconnection of $\xi$ with zero principal homomorphism
(i.e. it is $C^\infty(M)$-bilinear) for every $(\varphi,\omega)\in\Omega^1(\xi,\mathfrak{D})\times\Omega^1(M,\mathfrak{D})$.

For every triple $(\varphi,d,\omega)\in\Omega^1(\xi,C^\infty(M))\times
\mathfrak{D}\times\Omega^1(M,\mathfrak{D})$ we define
\begin{equation}
\label{fundamental}
\nabla^{(\varphi,d,\omega)}=\partial\varphi \cdot d+\varphi\cdot \omega.
\end{equation}
As $\partial\varphi\cdot d$ is a $\mathfrak{D}$-pseudoconnection
of $\xi$ with principal
homomorphism $\varphi\cdot d$, and $\varphi \cdot\omega$
is $C^\infty(M)$-bilinear, we have that $\nabla^{(\varphi,d,\omega)}$
is a $\mathfrak{D}$-pseudoconnection of $\xi$ with principal homomorphism
$\varphi\cdot d$.
A $\mathfrak{D}$-pseudoconnection
equals to $\nabla^{(\varphi,d,\omega)}$ for
some $(\varphi,d,\omega)\in\Omega^1(\xi,C^\infty(M))\times
\mathfrak{D}\times\Omega^1(M,\mathfrak{D})$ will be refereed to as
a {\em fundamental $\mathfrak{D}$-pseudoconnection of $\xi$}.

Clearly the sum of finitely many fundamental $\mathfrak{D}$-pseudoconnections
$\{\nabla^{(\varphi_i,d_i,\omega_i)}\}_{i=1}^k$ of $\xi$ is
a $\mathfrak{D}$-pseudoconnection of $\xi$ with principal homomorphism
$\sum_{i=1}^k\varphi_i\cdot d_i$.
Let us prove that these sums exhaust all possible $\mathfrak{D}$-pseudoconnections of
the $k$-dimensional trivial bundle $\epsilon^k_M=M\times \mathbb{R}^k$ over $M$.

\begin{prop}
\label{laforma}
Every $\mathfrak{D}$-pseudoconnection of
$\epsilon^k_M$ is the sum of $k$ fundamental $\mathfrak{D}$-pseudoconnections.
\end{prop}

\begin{proof}
Let $\nabla$ be a $\mathfrak{D}$-pseudoconnection of $\epsilon_M^k$.
Fix a base $\{e_1,\cdots, e_k\}$ of $\Omega^1(\epsilon^k_M)$
as a $C^\infty(M)$-module.
Then, there are $\varphi_1,\cdots, \varphi_k\in \Omega^1(\epsilon_M^k,\mathfrak{D})$
such that
$$
s=\sum_{i=1}^k\varphi_i(s)e_i,
\,\,\,\,\forall s\in\Omega^1(\epsilon^k_M).
$$
Let $P$ be the principal homomorphism of $\nabla$.
For all $i\in\{1,\cdots, k\}$ we define
$d_i=P(e_i)\in \mathfrak{D}$
and $\omega_i\in\Omega^1(M,\mathfrak{D})$ by
$\omega_i(X)=\nabla_Xe_i$, $\forall X\in \mathcal{X}$.
It follows from the properties of $\mathfrak{D}$-pseudoconnections that
$\nabla=\sum_{i=1}^k\nabla^{(\varphi_i,d_i,\omega_i)}$.
\end{proof}

For $k=1$ the above proposition gives the following example.

\begin{example}
\label{ex1}
Let $\epsilon^1_M$ be the trivial line bundle over $M$. As
$\Omega^0(\epsilon^1_M)$ is canonically isomorphic to $C^\infty(M)$ every $\mathfrak{D}$-pseudoconnection $\nabla$
of $\epsilon_M^1$ has the form
$$
\nabla_Xf=X(f)\cdot A+f\cdot \omega(X),
\,\,\,\,\,\,\,\,\,\forall (X,f)\in \mathcal{X}\times C^\infty(M),
$$
for some $A\in \mathfrak{D}$ and some
$\omega:\Omega^1(M,\mathfrak{D})$.
\end{example}

There is a well known method to construct ordinary connections
called pullback. Let us perform it but for
$\mathfrak{D}$-pseudoconnections, the only difference being the target module $\mathfrak{D}$.

First we recall some basic definitions.
A {\em bundle map} between
vector bundles $\Pi:\xi\to M$ and $\overline{\Pi}:\overline{\xi}\to\overline{M}$
over differentiable manifolds $M$ and $\overline{M}$ respectively is
a $C^\infty$ map $F:\xi\to\overline{\xi}$ which carries each vector space
$\xi_p$ isomorphically onto one of the vector spaces $\overline{\xi}_{\overline{p}}$.
Note that $F$ induces two $C^\infty$ maps
$f:M\to\overline{M}$, $f(p)=\overline{p}$, and
$F_*:\Omega^0(\xi)\to\Omega^0(\overline{\xi})$,
$F_*s=\overline{s}$, where for all $s\in \Omega^0(\xi)$ the section $\overline{s}\in\Omega^0(\overline{\xi})$ is the unique one satisfying
$$
F(s(p))=\overline{s}(f(p)),
\,\,\,\,\,\forall p\in M.
$$
We shall assume hereafter that the induced map $f$ above
is a diffeomorphism for all bundle map $F$.
With this assumption we have that the induced map $F_*$ satisfies
$$
F_*(\varphi s)=(\varphi\circ f^{-1})F_*s,
\,\,\,\,\,\,\,\forall s\in\Omega^0(\xi),\,\,\forall \varphi\in C^\infty(M).
$$
The basic example comes from the derivative
$Df:TM\to T\overline{M}$ of a $C^\infty$ diffeomorphism $f:M\to\overline{M}$.
In such a case we write $f_*$ instead of $(Df)_*$ for simplicity.

Now, let $F:\xi\to\overline{\xi}$ be a bundle map
between vector bundles $\xi,\overline{\xi}$ over $M$ and $\overline{M}$ respectively. Let $\overline{\Phi}: \overline{\mathfrak{D}}\to\mathfrak{D}$
be a map from a $C^\infty(\overline{M})$ module $\overline{\mathfrak{D}}$
into a $C^\infty(M)$ module $\mathfrak{D}$.
Given positive integers $k,t$ and a map $\mathfrak{A}:\mathcal{X}\times
\overset{(k)}{\cdots}\times\mathcal{X}\times\Omega^0(\xi)\times \overset{(t)}{\cdots}
\times \Omega^0(\xi)\to \overline{\mathfrak{D}}$
we define the {\em pullback
of $\mathfrak{A}$ under $(F,\overline{\Phi})$} as the map
$(F,\overline{\Phi})^*(\mathfrak{A}):\mathcal{X}\times
\overset{(s)}{\cdots}\times\mathcal{X}\times\Omega^0(\xi)\times \overset{(t)}{\cdots}
\times \Omega^0(\xi)\to \mathfrak{D}$ defined by
$$
(F,\overline{\Phi})^*(\mathfrak{A})(X_1,\cdots, X_k,s_1,\cdots,s_t)=
\overline{\Phi}(\mathfrak{A}(f_*X_1,\cdots,f_*X_k,F_*s_1,
\cdots,F_*s_t),
$$
for all
$X_1,\cdots X_k\in \mathcal{X}$ and $s_1,\cdots,s_t\in\Omega^0(\xi)$.

With these notations and definitions we have the following

\begin{prop}
\label{pull}
If $\overline{\Phi}$ above satisfies
\begin{description}
\item[(P1)]
$\overline{\Phi}(\overline{d}_1+\overline{d}_2)=
\overline{\Phi}(\overline{d}_1)+\overline{\Phi}(\overline{d}_2)$;
\item[(P2)]
$\overline{\Phi}((\varphi\circ f^{-1})\overline{d}_1)=
\varphi \overline{\Phi}(\overline{d}_1)$, $\forall \overline{d}_1,\overline{d}_2\in\overline{\mathfrak{D}}$,
$\forall \varphi\in C^\infty(M)$,
\end{description}
then
the pullback $(F,\overline{\Phi})^*(\overline{\nabla})$
of a $\overline{\mathfrak{D}}$-pseudoconnection $\overline{\nabla}$ with principal homomorphism $\overline{P}$ of $\overline{\xi}$ is
a $\mathfrak{D}$-pseudoconnection with principal homomorphism
$P=\overline{\Phi}\circ \overline{P}\circ F_*$ of $\xi$.
\end{prop}

\begin{proof}
We see clearly that $(F,\overline{\Phi})^*(\overline{\nabla})$ is $\mathbb{R}$-linear. Moreover,
\[
\begin{array}{ccl}
(F,\overline{\Phi})^*(\overline{\nabla})_{\varphi X}s
& = &
\overline{\Phi}(\overline{\nabla}_{f_*(\varphi X)}F_*s)\\
& = &
\overline{\Phi}((\varphi\circ f^{-1})\overline{\nabla}_{f_*X}F_*s)\\
& = &
\varphi\cdot \overline{\Phi}(\overline{\nabla}_{f_*X}F_*s)\\
& = &
\varphi\cdot (F,\overline{\Phi})^*(\overline{\nabla})_Xs
\end{array}
\]
and
\[
\begin{array}{ccl}
(F,\overline{\Phi})^*(\overline{\nabla})_{X}(\varphi s)
& = &
\overline{\Phi}(\overline{\nabla}_{f_*X}F_*(\varphi s))\\
& = &
\overline{\Phi}(\overline{\nabla}_{f_*(\varphi X)}(\varphi \circ f^{-1})\cdot F_*s)\\
& = &
\overline{\Phi}( f_*X(\varphi \circ f^{-1})\cdot \overline{P}(F_*s)+
(\varphi\circ f^{-1})\cdot \overline{\nabla}_{f_*(\varphi X)}F_*s)\\
& = &
\overline{\Phi}( (X(\varphi) \circ f^{-1})\cdot \overline{P}(F_*s)+
(\varphi\circ f^{-1})\cdot \overline{\nabla}_{f_*(\varphi X)}F_*s)\\
& = &
X(\varphi) \overline{\Phi}(\overline{P}(F_*s))+\varphi\cdot \overline{\Phi}(\overline{\nabla}_{f_*X}F_*s)\\
& = &
X(\varphi) P(s)+\varphi\cdot (F,\overline{\Phi})^*(\overline{\nabla})_Xs
\end{array}
\]
for all $(X,s)\in \mathcal{X}(M)\times \Omega^0(\xi)$ and all $\varphi \in C^\infty(M)$.

So, $(F,\overline{\Phi})^*(\overline{\nabla})$
is a $\mathfrak{D}$-pseudoconnection with principal homomorphism $P$.
\end{proof}

\subsection{A curvature for $\mathfrak{D}$-pseudoconnections}
\label{lacurva}
The curvature of an ordinary connection $\nabla$ of a
vector bundle $\xi$ over a manifold $M$ is
defined in classical geometry as
the map $R:\mathcal{X}\times\mathcal{X}\times\Omega^0(\xi)\to\Omega^0(\xi)$,
\begin{equation}
\label{ordinary}
R(X,Y)s=\nabla_X\nabla_Ys-\nabla_Y\nabla_Xs-\nabla_{[X,Y]}s,
\,\,\,\,\,\,\,\,\,\,\,\forall (X,Y,s)\in \mathcal{X}\times\mathcal{X}\times\Omega^0(\xi).
\end{equation}
It follows that $R$ is a tensor field
(i.e. $C^\infty(M)$-linear in its three variables) which is skew-symmetric
in the first two variables and satisfies
the Bianchi inequality whenever $\xi=TM$ and $\nabla$ is torsion free.
This curvature was extended to ordinary pseudoconnections
in \cite{a2} by setting
$$
R(X,Y)s =  \nabla_X\nabla_Y(Ps)-\nabla_Y\nabla_X(Ps)
- \nabla_XP(\nabla_Ys)+P\nabla_X\nabla_Ys+
$$
$$
\nabla_YP(\nabla_Xs)-
P\nabla_Y\nabla_Xs-
P\left(\nabla_{[X,Y]}P(s)\right),
\,\,\,\,\,\,\forall (X,Y,s)\in \mathcal{X}\times\mathcal{X}\times\Omega^0(\xi),
$$
where $P$ is the principal homomorphism of $\nabla$.
We would like to extend it to
$\mathfrak{D}$-pseudoconnection,
but, unfortunately, expressions like $\nabla_X\nabla_Y(Ps)$ (say)
are meaningless for arbitrary modules $\mathfrak{D}$.

Here we define a curvature for certain $\mathfrak{D}$-pseudoconnections
depending on the module $\mathfrak{D}$.
To define it we fix differentiable manifold $M$ and an integer $k\geq 1$.
Define $\mathfrak{D}^k(M)$ as the set of all maps
from $d:\mathcal{X}\times\overset{(k)}{\cdots}\times \mathcal{X}\to C^\infty(M)$.
This set is clearly a $C^\infty(M)$-module
if equipped with the standard sum and the multiplication
$$
(\varphi d)(X_1,\cdots,X_k)=\varphi d(X_1,\cdots,X_k),
\,\,\,\,\,\forall X_1,\cdots,X_k\in \mathcal{X}(M), \forall \varphi\in C^\infty(M).
$$
Given $X\in \mathcal{X}$ and $d\in \mathfrak{D}^k(M)$
we define $X(d)\in \mathfrak{D}^k(M)$ by
$$
X(d)(X_1,\cdots,X_k)=X(d(X_1,\cdots,X_k)),
\,\,\,\,\,\,\,\forall X_1,\cdots,X_k\in \mathcal{X}.
$$

\begin{defi}
\label{curvature}
The {\em curvature} of a $\mathfrak{D}^k(M)$-pseudoconnection $\nabla$ of a vector bundle $\xi$ over $M$ is the map $R:\mathcal{X}\times\mathcal{X}\times \Omega^0(\xi)\to \mathfrak{D}^k(M)$ defined by
$$
R(X,Y)s=X(\nabla_Y s)-Y(\nabla_Xs)-\nabla_{[X,Y]}s,
\,\,\,\,\,\,\forall (X,Y,s)\in \mathcal{X}\times\mathcal{X}\times \Omega^0(\xi).
$$
Sometimes we write $R^\nabla$ to indicate dependence on $\nabla$.
\end{defi}

This definition has both advantages and disadvantages if compared with that in \cite{a2}.  For it is simpler than \cite{a2},
but, unfortunately, the curvature under such a definition is not a tensor in the third variable and never vanishes in the $2$-Riemannian case
(see Corollary \ref{never-vanish2}).

\begin{example}
\label{exxx}
A straightforward computation shows that
the curvature of the fundamental $\mathfrak{D}^k(M)$-pseudoconnection $\nabla^{(\varphi,d,\omega)}$
of $\xi$ defined in (\ref{fundamental}) is
$$
R^{\nabla^{(\varphi,d,\omega)}}=
(\omega_d-\omega)\circ\hat{\partial}\varphi+\varphi\cdot d\omega,
$$
where:
\begin{itemize}
\item
$\omega_d\in\Omega^1(\xi,\mathfrak{D}^k(M))$ is the evaluation $1$-form
$\omega_d(X)= X(d)$;
\item
$d\omega: \mathcal{X}\times\mathcal{X}\to\mathfrak{D}^k(M)$ is the differential of $\omega$,
$$
d\omega(X,Y)=X(\omega(Y))-Y(\omega(X))-\omega([X,Y]),
$$
and
\item
$\hat{\partial}\varphi:\mathcal{X}\times\mathcal{X}\times\Omega^0(\xi)\to \mathcal{X}$ is defined by
$$
\hat{\partial}\varphi_{(X,Y)}s=Y(\varphi(s))X-X(\varphi(s))Y.
$$
\end{itemize}
We can use this formula together with Proposition \ref{laforma}
in order to compute
the curvature of every $\mathfrak{D}^k(M)$-pseudoconnection
of the trivial bundle $\epsilon_M^k$.
\end{example}

Next we state some properties of the aforementioned curvature.
First observe that given $f\in C^\infty(M)$ we can apply the differential $df$ to any pair
of vector fields $(X,Y)\in \mathcal{X}\times\mathcal{X}$
by defining
\begin{equation}
\label{df}
df(X,Y)=df(Y)X-df(X)Y.
\end{equation}
From this we obtain a map
$df:\mathcal{X}\times\mathcal{X}\to\mathcal{X}$
which is clearly a $\mathcal{X}$-valued $2$-form.

We can extend (\ref{differential})
to include $P:\Omega^0(\xi)\to \mathfrak{D}^k(M)$ yielding
$\partial P:\mathcal{X}\times\Omega^0(\xi)\to\mathfrak{D}^k(M)$ defined by
\begin{equation}
\label{differential2}
\partial P_Xs=X(P(s)),
\,\,\,\,\,\,\,\,\forall (X,s)\in\mathcal{X}\times \Omega^0(\xi).
\end{equation}
Thus, for every $\mathfrak{D}^k(M)$-pseudoconnection
$\nabla$ with principal homomorphism $P$ of $\xi$ we can define
the $C^\infty(M)$-bilinear map
$\partial\nabla:\mathcal{X}\times\Omega^0(\xi)\to\mathfrak{D}^k(M)$ by
$$
\partial\nabla=\partial P-\nabla.
$$

Denote by
$\mathfrak{D}_0^k(M)$
the submodule of $\mathfrak{D}^k(M)$ consisting of those $d\in \mathfrak{D}^k(M)$ which are $C^\infty(M)$-linear in the first variable.
Note that $\mathfrak{D}^1_0(M)=\Omega^1(M)$.
The proof of the theorem below follows from straightforward computations which are left to the reader.

\begin{thm}
\label{property-curvature}
The curvature $R$ of a $\mathfrak{D}_0^k(M)$-pseudoconnection $\nabla$
of $\xi$ has the following properties
for all $(X,Y,Z,s)\in \mathcal{X}\times\mathcal{X}\times\mathcal{X}\times\Omega^0(\xi)$
and $f\in C^\infty(M)$.

\begin{enumerate}
\item
$R$ is $\mathbb{R}$-trilinear.
\item
$R(X,Y)s=-R(Y,X)s$.
\item
$R(f\cdot X,Y)s=f\cdot R(X,Y)s$.
\item
$R(X,Y)(f\cdot s)=\partial\nabla(df(X,Y),s)+f\cdot R(X,Y)s$.
\item
If $\xi=TM$ and $\nabla$ is torsion free, then
$$
R(X,Y)Z+R(Y,Z)X+R(Z,X)Y=
$$
$$
\partial\nabla(X,[Y,Z])+\partial\nabla(Y,[Z,X])+\partial\nabla(Z,[X,Y]).
$$
\end{enumerate}
\end{thm}

The first three properties listed above are the usual
ones of the classical curvature tensor.
The fourth one says that our curvature although not $C^\infty(M)$-linear
in the third variable
behaves like a pseudoconnection with principal part $\partial\nabla$
in such a variable.
The last property is nothing but a version of the classical Bianchi
identity in Riemannian geometry.

Now we explain how the curvature is modified by pullbacks.

\begin{prop}
\label{curvado}
Let $F:\xi\to\overline{\xi}$ be a bundle map
between vector bundles $\xi$ and $\overline{\xi}$ over
differentiable manifolds $M$ and $\overline{M}$
respectively.
Consider a map $\overline{\Phi}:\mathfrak{D}^k(\overline{M})\to\mathfrak{D}^k(M)$
satisfying {\bf (P1)} and
{\bf (P2)} of Proposition \ref{pull} and also the following additional property,
\begin{description}
\item[(P3)]
$X(\overline{\Phi}(\overline{d}))=\overline{\Phi}(f_*X(\overline{d}))$
for all $(X,\overline{d})\in\mathcal{X}(M)\times \mathfrak{D}^k(\overline{M})$, where $f:M\to\overline{M}$ is the diffeomorphism induced by $F$
in the base manifolds.
\end{description}
Then,
for all $\mathfrak{D}^k(\overline{M})$-pseudoconnection $\overline{\nabla}$
of $\overline{\xi}$ the curvature
$R^{\overline{\nabla}}$ of $\overline{\nabla}$
and the curvature $R^{(F,\overline{\Phi})^*(\overline{\nabla})}$
of the pullback $(F,\overline{\Phi})^*(\overline{\nabla})$ satisfy
$$
R^{(F,\overline{\Phi})^*(\overline{\nabla})}=(F,\overline{\Phi})^*(R^{\overline{\nabla}}).
$$
\end{prop}

\begin{proof}
This can be shown by a direct computation. Indeed,
\[
\begin{array}{ccl}
R^{(F,\overline{\Phi})^*(\overline{\nabla})}(X,Y)s
& = &
X((F,\overline{\Phi})^*(\overline{\nabla})_Ys)-
Y((F,\overline{\Phi})^*(\overline{\nabla})_Xs)-
(F,\overline{\Phi})^*(\overline{\nabla})_{[X,Y]}s)\\
& &
\\
& = &
X(\overline{\Phi}(\overline{\nabla}_{f_*Y}F_*s))-
Y(\overline{\Phi}(\overline{\nabla}_{f_*X}F_*s))-
\overline{\Phi}(\overline{\nabla}_{f_*[X,Y]}F_*s)\\
& &
\\
& \overset{{\bf (P3)}}{=} &
\overline{\Phi}(f_*X(\overline{\nabla}_{f_*Y}F_*s))-
\overline{\Phi}(f_*Y(\overline{\nabla}_{f_*X}F_*s))-
\overline{\Phi}(\overline{\nabla}_{[f_*X,f_*Y]}F_*s)\\
& &
\\
& \overset{{\bf (P1)}}{=} &
\overline{\Phi}(f_*X(\overline{\nabla}_{f_*Y}F_*s)-
f_*Y(\overline{\nabla}_{f_*X}F_*s)-
\overline{\nabla}_{[f_*X,f_*Y]}F_*s)\\
& &
\\
& = &
\overline{\Phi}(R^{\overline{\nabla}}(f_*X,f_*Y)F_*s)\\
& &
\\
& = &
(F,\overline{\Phi})^*(R^{\overline{\nabla}})(X,Y)s.
\end{array}
\]
\end{proof}

\subsection{Symmetry and compatibility}
Fix a differentiable manifold $M$ and an integer $k\geq 1$.
A $\mathfrak{D}^k(M)$-pseudoconnection $\nabla$ of $TM$ is
{\em compatible} with a homomorphism $P:\mathcal{X}\to \mathfrak{D}^k(M)$ if
$$
XP(Y)(X_ 1,\cdots, X_k)=\nabla_XY(X_1,\cdots,X_k)+\nabla_XX_1(Y,X_2,\cdots,X_k),
$$
for all $X,Y,X_1,\cdots,X_k\in\mathcal{X}$.

{\em A $\mathfrak{D}^k_0(M)$-pseudoconnection is compatible only with
its own principal homomorphism}.
Indeed, suppose that $\nabla$ is a $\mathfrak{D}^k_0(M)$-pseudoconnection compatible with a homomorphism $P$.
Take $X,Y,Z\in \mathcal{X}$ and $f\in C^\infty(M)$.

On the one hand, as $\nabla$ is compatible with $P$ we have
\[
\begin{array}{ccl}
\nabla_XfY(X_1,\cdots,X_k)+\nabla_XX_1(fY,X_2,\cdots, X_k)
& = &
XP(fY)(X_1,\cdots,X_k)\\
& = &
X(fP(Y)(X_1,\cdots, X_k))\\
& = &
X(f)P(Y)(X_1,\cdots, X_k)+\\
& &
fXP(Y)(X_1,\cdots, X_k)\\
& = &
X(f)P(Y)(X_1,\cdots, X_k)+\\
& &
f\nabla_XY(X_1,\cdots, X_k)+\\
& &
f\nabla_XX_1(Y,X_2,\cdots, X_k).
\end{array}
\]
On the other hand, if $P^\nabla$ is the principal homomorphism of $\nabla$,
then
\[
\begin{array}{ccl}
\nabla_XfY(X_1,\cdots,X_k)+\nabla_XX_1(fY,X_2,\cdots, X_k)
& = &
X(f)P^\nabla(Y)(X_1,\cdots, X_k)+\\
& &
f\nabla_XY(X_1,\cdots, X_k)+\\
& &
f\nabla_XX_1(Y,X_2,\cdots, X_k).
\end{array}
\]
Then,
$$
X(f)P(Y)(X_1,\cdots, X_k)=X(f)P^\nabla(Y)(X_1,\cdots, X_k),
\forall X,Y,X_1,\cdots, X_k\in\mathcal{X},\varphi\in C^\infty(M),
$$
so we have $P=P^\nabla$.
It is by this reason that, in the case of $\mathfrak{D}^k_0(M)$-pseudoconnections $\nabla$ of $TM$, we shall say that $\nabla$ is compatible
without any reference to the homomorphism $P$.

On the other hand, a homomorphism $P:\mathcal{X}\to\mathfrak{D}^k(M)$ is {\em symmetric}
if
$$
P(X)(X_1,\cdots, X_k)=P(X_1)(X,X_2,\cdots, X_k),
\,\,\,\,\,\forall X,X_1,\cdots, X_k\in\mathcal{X}.
$$

Every homomorphism $P$ compatible with a $\mathfrak{D}^k(M)$-pseudoconnection $\nabla$ of $TM$ is symmetric.
For if $X,X_1,\cdots, X_k,Y\in\mathcal{X}$ then
\[
\begin{array}{ccl}
YP(X)(X_1,\cdots, X_k)
& = &
\nabla_YX(X_1,\cdots, X_k)+\nabla_YX_1(X,X_2,\cdots, X_k)\\
& = &
\nabla_YX_1(X,X_2,\cdots, X_k)+\nabla_YX(X_1,\cdots, X_k)\\
& = &
YP(X_1)(X,X_2,\cdots, X_k)
\end{array}
\]
and so $YP(X)(X_1,\cdots, X_k)=YP(X_1)(X,X_2,\cdots, X_k)$
therefore $P$ is symmetric since
$X,X_1,\cdots, X_k,Y\in\mathcal{X}$ are arbitrary.
In particular, the principal homomorphism of a compatible
$\mathfrak{D}^k_0(M)$-pseudoconnection  of $TM$ is symmetric.
The converse is true also due to the following version
of the classical Levi-Civita Theorem in Riemannian geometry.

\begin{prop}
\label{levi-civita}
Every symmetric homomorphism $P:\mathcal{X}\to\mathfrak{D}^k_0(M)$
is the principal homomorphism of a unique
torsion free compatible $\mathfrak{D}^k_0(M)$-pseudoconnection of $TM$.
\end{prop}

\begin{proof}
The proof is similar to that of the Levi-Civita Theorem \cite{dcm} except that here we ignore the vector field components $(X_2,\cdots, X_k)$
of $(X_ 1,\cdots, X_k)$.
More precisely, the hypotheses of the proposition imply that if the desired
pseudoconnection $\nabla^P$ exists then it must satisfy
$$
\nabla^{P}_XY(X_1,\cdots, X_k)=\frac{1}{2}\{XP(Y)(X_1,\cdots, X_k)
+ YP(X_1)(X, X_2,\cdots, X_k)-
$$
$$
X_1P(X)(Y,X_2,\cdots, X_k)+
P([X,Y])(X_1,\cdots, X_k) +
$$
$$
P([X_1,X])(Y,X_2,\cdots, X_k) - P([Y,X_1])(X,
X_2,\cdots, X_k)\},
$$
for all $X,X_1,\cdots, X_k,Y\in \mathcal{X}$.
It is straightforward to check that such a
$\nabla^P$ satisfies the properties required in the proposition.
\end{proof}

In the particular case of type $(2,0)$ symmetric tensor fields
$h$ on $M$ we have the following example.

\begin{example}
\label{superex}
Every type $(2,0)$ tensor field $h$ of $M$ induces a homomorphism $P^h: \mathcal{X}\to\Omega^1(M)$, $P^h(X)(Y)=h(X,Y)$
for all $X,Y\in\mathcal{X}$.
If $h$ is symmetric then $P^h$ does hence, by Proposition \ref{levi-civita},
there is a unique torsion free compatible
$\Omega^1(M)$-pseudoconnection $\nabla^h$ of $TM$ with principal homomorphism
$P^h$ (note that $\Omega^1(M)=\mathfrak{D}^1_0(M)$).
\end{example}

We shall use this example later to motivate the definition of the $2$-Riemannian pseudoconnection associated to a $2$-Riemannian metric.

\subsection{Remarks on $\mathfrak{D}$-pseudoconnections with zero curvature}
The following result characterizes the $\mathfrak{D}^k_0(M)$-pseudoconnections
with zero curvature map and prescribed principal homomorphism.
Let $\xi$ be a vector bundle over a differentiable manifold $M$ and
let $k$ be an integer greater than $1$.

\begin{thm}
\label{ch-flat}
For every
homomorphism $P:\Omega^0(\xi)\to \mathfrak{D}_0^k(M)$
the differential $\partial P$ of $P$ in (\ref{differential2}) is
the unique $\mathfrak{D}^k_0(M)$-pseudoconnection of $\xi$ with
zero curvature and principal homomorphism $P$.
\end{thm}

\begin{proof}
It is not difficult to prove that
$\partial P$ is a $\mathfrak{D}^k_0(M)$-pseudoconnection
with principal homomorphism $P$ of $\xi$.
On the other hand, for all $X,Y\in \mathcal{X}$, $d,d'\in \mathfrak{D}^k(M)$ and $\varphi\in C^\infty(M)$ one has
\begin{itemize}
\item
$(X+Y)d=X(d)+Y(d)$ and $X(d+d')=X(d)+X(d')$;
\item
$X(\varphi\cdot d)=X(\varphi)\cdot d+\varphi\cdot X(d)$;
\item
$[X,Y](d)=X(Y(d))-Y(X(d))$.
\end{itemize}
So,
\[
\begin{array}{ccl}
X(\partial P_Ys)-Y(\partial P_Xs)
& = &
X(Y(P(s)))-Y(X(P(s)))\\
& = &
[X,Y](P(s))\\
& = &
\partial P_{[X,Y]}s,
\,\,\,\,\,\,\forall (X,Y,s)\in \mathcal{X}\times\mathcal{X}\times\Omega^0(\xi).
\end{array}
\]
Therefore $\partial P$ has zero curvature map.

Now suppose that $\nabla$ is another
$\mathfrak{D}_0^k(M)$-pseudoconnection of $\xi$ with both zero curvature and principal homomorphism $P$.
Then we have
$$
\partial\nabla(df(X,Y),s)=0,
\,\,\,\,\,\,\,\forall (X,Y,s)\in \mathcal{X}\times\mathcal{X}\times\Omega^0(\xi),\forall f\in
C^\infty(M)
$$
by Theorem \ref{property-curvature}-(4).
But $df:\mathcal{X}\times\mathcal{X}\to\mathcal{X}$ is
onto when restricted to any coordinate neighborhood of $M$.
As $\partial\nabla$ is a tensor we conclude that
$\partial\nabla=0$ which is equivalent to $\nabla=\partial P$.
\end{proof}

\begin{example}
\label{exxxx}
Consider the fundamental $\mathfrak{D}^k(M)$-pseudoconnection
$\nabla^{(\varphi,d,\omega_d)}$ where $\varphi\in\Omega^1(\xi,\mathfrak{D}^k(M))$ and $\omega_d$ is the evaluation form
associated to $d\in\mathfrak{D}^k(M)$ (see Example \ref{exxx}).
Such a $\mathfrak{D}^k(M)$-pseudoconnection has zero curvature since
$\omega_d$ is exact (i.e. has zero differential) for all $d\in\mathfrak{D}^k(M)$.
Therefore $\nabla^{(\varphi,d,\omega_d)}=\partial P^{\nabla^{(\varphi,d,\omega_d)}}$
by Theorem \ref{ch-flat}.
This last identity can be proved also with a direct computation.
\end{example}

Related to this example we have the following corollary.

\begin{clly}
The fundamental $\mathfrak{D}^k(M)$-pseudoconnection
$\nabla^{(\varphi,d,\omega)}$ of $\xi$ has
zero curvature if and only if $\varphi\cdot \omega=\varphi\cdot \omega_d$.
\end{clly}

\begin{proof}
If $\varphi\cdot\omega=\varphi\cdot\omega_d$, then
$$
\nabla^{(\varphi,d,\omega)}=
\partial\varphi\cdot d+\varphi\cdot\omega=
\partial\varphi\cdot d+\varphi\cdot\omega_d=
\nabla^{(\varphi,d,\omega_d)}
$$
thus $\nabla^{(\varphi,d,\omega)}$ has zero curvature
by Example \ref{exxxx}.

Conversely, if $\nabla^{(\varphi,d,\omega)}$ has zero curvature
then $\nabla^{(\varphi,d,\omega)}=\partial P^{\nabla^{(\varphi,d,\omega)}}$ by Theorem \ref{ch-flat}.
As
$\partial(\varphi\cdot d)=\partial\varphi\cdot d+\varphi\cdot\omega_d$
for all $(\varphi,d)\in\Omega^1(\xi,C^\infty(M))\times \mathfrak{D}^k(M)$ we have $
\nabla^{(\varphi,d,\omega)}=\partial\varphi\cdot d+\varphi\cdot\omega_d$
since
$P^{\nabla^{(\varphi,d,\omega)}}=\varphi\cdot d$.
But $
\nabla^{(\varphi,d,\omega)}=\partial\varphi\cdot d+\varphi\cdot
\omega$
by definition, so $\varphi\cdot\omega=\varphi\cdot\omega_d$.
\end{proof}

Another corollary is the following.

\begin{clly}
\label{never-vanish}
$\nabla=0$ is the unique compatible
$\mathfrak{D}_0^k(M)$-pseudoconnection with zero curvature of $\xi$.
\end{clly}

\begin{proof}
Let $\nabla$ be a
$\mathfrak{D}_0^k(M)$-pseudoconnection of $\xi$ with zero curvature.
By Theorem \ref{ch-flat} we have that if $P$ is the principal homomorphism
of $\nabla$ then
$$
\nabla_XY=X(P(Y)),
\,\,\,\,\,\,\,\,\forall X,Y\in\mathcal{X}.
$$
Thus, if $\nabla$ were compatibility we would have
\[
\begin{array}{ccl}
\nabla_XY(X_1,\cdots, X_k)
& = &
XP(Y)(X_1,\cdots, X_k)\\
& = &
\nabla_XY(X_1,\cdots, X_k)+\nabla_XX_1(Y,X_2,\cdots, X_k).
\end{array}
\]
Therefore $\nabla_XX_1(Y,X_2,\cdots, X_k)=0$ for all $X,X_1,\cdots, X_k,Y\in \mathcal{X}$ so $\nabla=0$.
\end{proof}

\subsection{The Koszul derivative}
There are cases where a
$\mathfrak{D}$-pseudoconnection $\nabla$ with principal
homomorphism $P$ of a vector bundles $\xi$
over a manifold $M$ admits a factorization,
$$
\nabla=P\circ D,
$$
for some map
$D:\mathcal{X}\times\mathcal{X}\to\mathcal{X}$.
In such a case we shall say that $D$ is a
{\em Koszul derivative}
of $\nabla$.
This definition can be found in \cite{k} but when $\mathfrak{D}=\Omega^1(M)$ and $P$
has the form $P(X)(Y)=\Phi(X,Y)$ for some map $\Phi:\mathcal{X}\times\mathcal{X}\to C^\infty(M)$.

\begin{example}
\label{example1}
A $\mathfrak{D}$-pseudoconnection $\nabla$ of the trivial line bundle
$\epsilon^1_M$ has Koszul derivatives if and only if
there are $\tau\in \Omega^1(M)$ and $A\in \mathfrak{D}$ such that
$$
\nabla_Xf=(X(f)+f\cdot\tau(X))\cdot A,
\,\,\,\,\,\,\,\forall (X,f)\in \mathcal{X}\times \Omega^0(\epsilon^1_M).
$$
In such a case there is a unique Koszul derivative $D:\mathcal{X}\times\Omega^0(\epsilon_M^1)\to\Omega^0(\epsilon_M^1)$ defined by
$D_Xf=X(f)+f\cdot \tau(X)$.
\end{example}

Let us prove the above assertion.
If $\nabla$ has the desired form then $P(f)=f\cdot A$ is its principal homomorphism
hence $D$ as above is a Koszul derivative of $\nabla$.
Conversely, suppose that $\nabla$ has a Koszul derivative $D$.
We know from Example \ref{ex1} that $\nabla$ has the form,
$$
\nabla_Xf=X(f)\cdot A+f\cdot\omega(X),
$$
for some $A\in \mathfrak{D}$ and some $\mathfrak{D}$-valued $1$-form
$\omega:\mathcal{X}\to\mathfrak{D}$.
In such a case $P(f)=fA$ is the principal homomorphism of $\nabla$
which is clearly injective since $C^\infty(M)$ is an unitary
ring.
Set $\tau(X)=D_X1$ where $1$ here is the constant map $p\mapsto 1$.
As $P$ is injective we have that $\tau\in\Omega^1(M)$.
The following computation
$$
\tau(X)\cdot A
=
P(\tau(X))
=
P(D_X1)
=
\nabla_X1
=
X(1)\cdot A+ 1\cdot\omega(X)
=
\omega(X)
$$
shows that $\omega(X)=\tau(X)\cdot A$. Replacing in the expression of $\nabla$ above we get
the desired form for $\nabla$.

\begin{example}
In Example \ref{example1} if $\mathfrak{D}$ has two distinct elements $d_1,d_2$ and $\omega\in \Omega^1(M)$ is non-zero,
then $ \nabla:\mathcal{X}\times\Omega^0(\epsilon_M^1)\to \mathfrak{D}$
defined by $\nabla_Xf=X(f)\cdot d_1+f\cdot\omega(X)\cdot d_2$
is a $\mathfrak{D}$-pseudoconnection of $\epsilon_M^1$ without
Koszul derivatives.
\end{example}

These examples motivate the question whether Koszul derivatives
exist for a given pseudoconnection $\nabla$.
Obviously a necessary condition for the
existence of such a derivative is that
$Im(\nabla)\subset Im(P)$, where $Im(\cdot)$ stands for the image operator.
It is also obvious that $Im(\nabla)\subset Im(P)$ is a
sufficient condition when $P$ is injective
and, in such a case, $D=P^{-1}\circ \nabla$ is the unique
Koszul derivative of $\nabla$.
However it may happen not only that
$Im(\nabla)\not\subset Im(P)$ but also that $Im(\nabla)\cap Im(P)=\{0\}$
(see for instance Theorem \ref{new-proof}).
All of this show the relationship between existence of Koszul derivatives and the injectivity of the principal homomorphism.
Let us give two short results exploring further such a relation.
Hereafter we denote by $Ker(\cdot)$ the kernel operation.

\begin{prop}
If $D$ is a Koszul derivative of a pseudoconnection with principal homomorphism $P$ of a vector bundle $\xi$ over $M$, then
$D_{X+Y}s-D_Xs-D_Ys,D_{fX}s-fD_Xs,D_X(s+ s')-D_X-D_Xs',
D_X(fs)-X(f)s-fD_Xs\in Ker(P)$
for all $X,Y\in \mathcal{X}$, $s,s'\in \Omega^0(\xi)$, $f\in C^\infty(M)$.
In particular, if $P$ is injective, then $D$ is not only unique
but also an ordinary connection of $\xi$.
\end{prop}

\begin{prop}
Every pseudoconnection of $TM$ having a unique Koszul derivative
has injective principal homomorphism.
In particular, such a Koszul derivative is an ordinary connection of $TM$.
\end{prop}

\begin{proof}
Let $\nabla$ be a pseudoconnection of $TM$ having a unique Koszul
derivative $D$.
If $P$ is the principal homomorphism and $T: \mathcal{X}\times\mathcal{X}\to Ker(P)$ is $C^\infty(M)$-bilinear,
then $P(D_XY+T(X,Y))=\nabla_XY$, for all $X,Y\in \mathcal{X}$.
Therefore $D+T$ is also a Koszul derivative of $\nabla$, and so,
$T=0$ by the uniqueness.
Now, let $Z\in Ker(P)$ and $h$ be a Riemannian metric of $M$.
Then, $T$ defined by $T(X,Y)=h(X,Y)\cdot Z$, $\forall X,Y\in\mathcal{X}$,
is $C^\infty(M)$-bilinear with values in $Ker(P)$.
Then, $T=0$ and so $Z=0$. Therefore $Ker(P)=0$ and so $P$ is injective.
\end{proof}

Now suppose that $\nabla$ is a compatible $\mathfrak{D}^k(M)$-pseudoconnection of $TM$ with
injective principal homomorphism $P$.
Suppose in addition that $\nabla$ has a Koszul derivative $\theta$.
It follows that $\theta$ is an ordinary connection of $TM$ hence
$\theta$ has a curvature tensor $R^\theta$ defined in (\ref{ordinary}).
Let us express the curvature of $\nabla$, $R^\nabla$,
as a function of $R^\theta$:
\[
\begin{array}{ccl}
R^\nabla(X,Y)Z(X_1,\cdots, X_k)
& = &
X(\nabla_YZ(X_1,\cdots, , X_k))-\\
& &
Y(\nabla_XZ(X_1,\cdots, , X_k))-
\nabla_{[X,Y]}Z(X_1,\cdots, X_k)\\
& = &
X(P(\theta_YZ)(X_1,\cdots, , X_k))-\\
& &
Y(P(\theta_XZ)(X_1,\cdots, , X_k))-\\
& &
\nabla_{[X,Y]}Z(X_1,\cdots, X_k)\\
& = &
(\nabla_X\theta_YZ-\nabla_Y\theta_XZ-
\nabla_{[X,Y]}Z)(X_1,\cdots, X_k)+\\
& &
\nabla_XX_1(\theta_YZ, X_2,\cdots, X_k)-\\
& &
\nabla_YX_1(\theta_XZ,X_2,\cdots, X_k)\\
& = &
P(\theta_X\theta_YZ-\theta_Y\theta_XZ-\theta_{[X,Y]}Z)(X_1,\cdots, X_k)+\\
& &
\nabla_XX_1(\theta_YZ, X_2,\cdots, X_k)-\\
& &
\nabla_YX_1(\theta_XZ,X_2,\cdots, X_k)
\end{array}
\]
thus we get the formula
$$
R^\nabla(X,Y)Z(X_1,\cdots, X_k)=
P(R^\theta(X,Y)Z)(X_1,\cdots, X_k)+
$$
$$
\nabla_XX_1(\theta_YZ, X_2,\cdots, X_k)
-\nabla_YX_1(\theta_XZ,X_2,\cdots, X_k),
\,\,\,\forall X,X_1,\cdots, X_k,Y,Z\in\mathcal{X}.
$$

Let us apply it to the situation described in Example \ref{superex}.
Indeed, let $h$ be a Riemannian metric of $M$ which is a type $(2,0)$ symmetric nondegenerate positive definite tensor field of $M$.
Hence there is a unique torsion free compatible
$\Omega^1(M)$-pseudoconnection $\nabla^h$ of $TM$ with principal homomorphism
$P^h$. As $h$ is a  Riemannian metric we have that $P^h$ is an isomorphism
hence $\nabla^h$ has a unique Koszul derivative $\theta^h$
(this is nothing but the Riemannian connection of $h$).
By the above formula the curvatures $R^{\nabla^h}$ and $R^h=R^{\theta^h}$ of $\nabla^h$ and $h$ respectively are related by
$$
R^{\nabla^h}(X,Y)Z(W)=
h(R^h(X,Y)Z,W)+h(\theta^h_XW,\theta^h_YZ)-h(\theta^h_YW,\theta^h_XZ),
$$
$\forall X,Y,Z,W\in\mathcal{X}$.

\section{The $2$-Riemannian pseudoconnection and curvature}
\label{$2$-Riem-pseudo}

\noindent
Motivated by the homomorphism $P^h$ in Example \ref{superex} we associate to every $2$-Riemannian metric
$g$ a map
$P^g: \mathcal{X}\to \mathfrak{D}_0^2(M)$ defined by
$$
P^g(X)(Y,Z)=g(X,Y/Z).
$$
It follows from the definition of $2$-inner products that
$P^g$ is symmetric. So,
by Proposition \ref{levi-civita},
there is a unique torsion free compatible $\mathfrak{D}^2_0(M)$-pseudoconnection $\nabla^g$ of $TM$ with principal
homomorphism $P^g$. Such a $\nabla^g$ will be refereed to as
the {\em $2$-Riemannian pseudoconnection of $g$}.
For later application we quote the formula of $\nabla^g$:
\begin{equation}
\label{pseudoconnection}
\nabla^{g}_XY(Z,W)=\frac{1}{2}\{Xg( Y,Z/W) + Yg( Z,X/W) - Zg( X,Y/W)+
\end{equation}
$$
g( [X,Y],Z/W) + g([Z,X],Y/W) - g([Y,Z],X/W)\},
$$
for all $X,Y,Z,W\in \mathcal{X}$.

By a {\em $2$-Riemannian pseudoconnection} we mean a pseudoconnection
equals to $\nabla^g$ for some $2$-Riemannian metric $g$.
The {\em curvature of a $2$-Riemannian metric $g$} is the curvature map
of $\nabla^g$.

\begin{clly}
\label{never-vanish2}
Every $2$-Riemannian metric
has non zero curvature.
\end{clly}

\begin{proof}
This follows from Corollary \ref{never-vanish} since every $2$-Riemannian pseudoconnection is non-zero.
\end{proof}

\subsection{Invariance by $2$-isometries}
In this subsection we obtain a $2$-Riemannian version of the classical
invariance by isometries of the Riemannian connection and curvature \cite{lee}.

Consider a diffeomorphism $f: M\to \overline{M}$ between differentiable manifolds
$M$ and $\overline{M}$.
Then, the derivative $Df:TM\to T\overline{M}$ is a bundle map
with induced map $f$ on the base manifolds $M$ and $\overline{M}$.
In such a case we have the notation $(Df)_*=f_*$.
On the other hand, $f$ induces for all integer $k\geq 1$ a natural map
$\overline{\Phi}_f: \mathfrak{D}^k(\overline{M})\to\mathfrak{D}^k(M)$
defined by
$$
\overline{\Phi}_f(\overline{d})(X_1, \cdots, X_k)=
\overline{d}(f_*X_1,\cdots, f_*X_k)\circ f,
$$
$\forall X_1,\cdots, X_k\in \mathcal{X}(M)$ ,$\forall \overline{d}\in\mathfrak{D}(\overline{M})$.
It is not difficult to see that
$\overline{\Phi}_f$ satisfies {\bf (P1)} and {\bf (P2)} in Proposition \ref{pull}.
Moreover, as
\[
\begin{array}{ccl}
X(\overline{\Phi}_f(\overline{d}))(X_1,\cdots, X_k)
& = &
X(\overline{\Phi}_f(\overline{d})(X_1, \cdots, X_k))\\
& = &
X(\overline{d}(f_*X_1, \cdots, f_*X_k)\circ f)\\
& = &
f_*X(\overline{d}(f_*X_1, \cdots, f_*X_k))\circ f\\
& = &
f_*X(\overline{d})(f_*X_1, \cdots,f_*X_k)\circ f\\
& = &
\overline{\Phi}_f(f_*X(\overline{d}))(X_1, \cdots, X_k),
\,\,\,\forall X_1, \cdots, X_k\in \mathcal{X}(M)
\end{array}
\]
one has that $\overline{\Phi}_f$ satisfies
{\bf (P3)} in Proposition \ref{curvado} too.

Now suppose that there are $2$-Riemannian metrics $g$ and $\overline{g}$ in $M$ and
$\overline{M}$ respectively so that
$f: M\to\overline{M}$ is a $2$-isometry.
Denote by $\nabla^g$ and $\nabla^{\overline{g}}$ the $2$-Riemannian
pseudoconnections of $g$ and $\overline{g}$ respectively.
As $\overline{\Phi}_f$ above satisfies {\bf (P1)} and {\bf (P2)} we have
for $k=2$
that the pullback $(Df,\overline{\Phi}_f)^*(\nabla^{\overline{g}})$ is a well defined $\mathfrak{D}_0^2(M)$-pseudoconnection of $TM$. Indeed we have the following

\begin{prop}
\label{invariance1}
$(Df,\overline{\Phi}_f)^*(\nabla^{\overline{g}})=\nabla^g$.
\end{prop}

\begin{proof}
It follows from
Proposition \ref{pull} that
the principal homomorphism
$P^{(Df,\overline{\Phi}_f)^*(\nabla^{\overline{g}})}$ of
$(Df,\overline{\Phi}_f)^*(\nabla^{\overline{g}})$ is
$\overline{\Phi}_f\circ P^{\overline{g}}\circ f_*$. Then,
$$
P^{(Df,\overline{\Phi}_f)^*(\nabla^{\overline{g}})}(X)=
\overline{\Phi}_f(P^{\overline{g}}(f_* X)),
\,\,\,\,\,\,\,\forall X\in\mathcal{X}(M)
$$
and so
\[
\begin{array}{ccl}
P^{(Df,\overline{\Phi}_f)^*(\nabla^{\overline{g}})}(X)(Y,Z)
& = &
P^{\overline{g}}(f_*X)(f_*Y,f_*Z)\circ f\\
& = &
\overline{g}(f_*X,f_*Y/f_*Z)\circ f,
\,\,\,\,\forall X,Y,Z\in\mathcal{X}(M).
\end{array}
\]
But $f$ is a $2$-isometry so
$$
\overline{g}(f_*X,f_*Y/f_*Z)\circ f=g(X,Y/Z)=P^g(X)(Y,Z),
\,\,\,\,\forall X,Y,Z\in\mathcal{X}(M).
$$
Replacing above we get
$$
P^{(Df,\overline{\Phi}_f)^*(\nabla^{\overline{g}})}=P^g.
$$

Next we observe that
\[
\begin{array}{ccl}
(Df,\overline{\Phi})^*(\nabla^{\overline{g}})_XY-
(Df,\overline{\Phi})^*(\nabla^{\overline{g}})_YX
& = &
\overline{\Phi}_f(\nabla^{\overline{g}}_{f_*X}f_*Y)-
\overline{\Phi}_f(\nabla^{\overline{g}}_{f_*Y}f_*X)\\
& = &
\overline{\Phi}_f(\nabla^{\overline{g}}_{f_*X}f_*Y-
\nabla^{\overline{g}}_{f_*Y}f_*X)\\
& = &
\overline{\Phi}_f(P^{\overline{g}}([f_*X,f_*Y])),
\,\,\,\,\,\,\forall X,Y\in\mathcal{X}(M)
\end{array}
\]
since $\nabla^{\overline{g}}$ is torsion free.
So,
\[
\begin{array}{ccl}
((Df,\overline{\Phi})^*(\nabla^{\overline{g}})_XY-
(Df,\overline{\Phi})^*(\nabla^{\overline{g}})_YX)(Z,W)
& = &
P^{\overline{g}}([f_*X,f_*Y])(f_*Z,f_*W)\circ f\\
& = &
\overline{g}([f_*X,f_*Y],f_*Z/f_*W)\circ f\\
& = &
g([X,Y],Z/W)\\
& = &
P^g([X,Y])(Z,W)\\
& = &
P^{(Df,\overline{\Phi}_f)^*(\nabla^{\overline{g}})}([X,Y])(Z,W),
\end{array}
\]
for all $X,Y,Z,W\in\mathcal{X}(M)$. Consequently
$(Df,\overline{\Phi})^*(\nabla^{\overline{g}})$ is torsion free.

On the other hand, we have that $\nabla^{\overline{g}}$ is compatible
so
\[
\begin{array}{ccl}
f_*X(P^{\overline{g}}(f_*Y)(f_*Z,f_*W))
& = &
\nabla^{\overline{g}}_{f_*X}f_*Y(f_*Z,f_*W)+
\nabla^{\overline{g}}_{f_*X}f_*Z(f_*Y,f_*W)\\
& = &
(\overline{\Phi}_f(\nabla^{\overline{g}}_{f_*X}f_*Y)(Z,W)+\\
& &
\overline{\Phi}_f(\nabla^{\overline{g}}_{f_*X}f_*Z)(Y,W))\circ f^{-1}\\
& = &
((Df,\overline{\Phi}_f)^*(\nabla^{\overline{g}})_XY(Z,W)+\\
& &
(Df,\overline{\Phi}_f)^*(\nabla^{\overline{g}})_XZ(Y,W))\circ f^{-1}.
\end{array}
\]
Hence
$$
f_*X(P^{\overline{g}}(f_*Y)(f_*Z,f_*W))\circ f=
(Df,\overline{\Phi}_f)^*(\nabla^{\overline{g}})_XY(Z,W)+
$$
$$
(Df,\overline{\Phi}_f)^*(\nabla^{\overline{g}})_XZ(Y,W)
$$
for all $X,Y,Z,W\in \mathcal{X}(M)$.
But
\[
\begin{array}{ccl}
f_*X(P^{\overline{g}}(f_*Y)(f_*Z,f_*W))\circ f
& = &
XP^g(Y)(Z,W)\\
& = &
XP^{(Df,\overline{\Phi}_f)^*(\nabla^{\overline{g})}}(Y)(Z,W)
\end{array}
\]
for all $X,Y,Z,W\in\mathcal{X}(M)$. Replacing above we get
$$
XP^{(Df,\overline{\Phi}_f)^*(\nabla^{\overline{g}})}(Y)(Z,W)=
(Df,\overline{\Phi}_f)^*(\nabla^{\overline{g}})_XY(Z,W)+
(Df,\overline{\Phi}_f)^*(\nabla^{\overline{g}})_XZ(Y,W),
$$
for all $X,Y,Z,W\in\mathcal{X}(M)$.
Therefore
$(Df,\overline{\Phi}_f)^*(\nabla^{\overline{g}})$ is compatible.
It follows that
$(Df,\overline{\Phi}_f)^*(\nabla^{\overline{g}})=\nabla^g$
by the uniqueness of the $2$-Riemannian pseudoconnection
(see Proposition \ref{levi-civita}).
\end{proof}

\begin{clly}
The curvatures $R^g$ and $R^{\overline{g}}$ of $g$ and $\overline{g}$ respectively are related by
$$
R^g(X,Y)Z(W,T)=R^{\overline{g}}(f_*X,f_*Y)f_*Z(f_*W,f_*T)\circ f,
\,\,\,\,\,\,\,\,\forall X,Y,Z,W,T\in\mathcal{X}(M).
$$
\end{clly}

\begin{proof}
This is a direct computation using propositions \ref{invariance1} and \ref{curvado}:
\[
\begin{array}{ccl}
R^g(X,Y)Z(W,T)
& = &
R^{(Df,\overline{\Phi}_f)^*(\nabla^{\overline{g}})}(X,Y)Z(W,T)\\
& = &
((Df,\overline{\Phi}_f)^*(R^{\nabla^{\overline{g}}})(X,Y)Z)(W,T)\\
& = &
\overline{\Phi}_f(R^{\overline{g}}(f_*X,f_*Y)f_*Z)(W,T)\\
& = &
R^{\overline{g}}(f_*X,f_*Y)f_*Z(f_*W,f_*T)\circ f.
\end{array}
\]
\end{proof}

\subsection{Non-existence of Koszul derivatives I}
In this subsection we shall prove that
every $2$-Riemannian pseudoconnection is Koszul derivative free.
The proof will involve some auxiliary definitions and notations.
Let $(M,g)$ be a $2$-Riemannian manifold.
We say that $X\in \mathcal{X}$ is {\em stationary} (with respect to $g$)
if
$$
Xg( Y,Z/ X)=
g([X,Y],Z/ X) + g(Y,[X,Z]/X), \,\,\,\forall Y,Z\in \mathcal{X}.
$$

\begin{rk}
\label{rk2}
$X\in\mathcal{X}$ is stationary with respect to $g$ if and only if
$$
Xg( Y,Y/X)=2g( [X,Y],Y/X),
\,\,\,\,\,\,\,\,\,\,\forall Y\in \mathcal{X}.
$$
\end{rk}

The name "stationary" is motivated by the fact that
if $X$ is non-singular stationary, then the map
$(u,v)\in TM\to g_p( u,v/X(p))$ defines a stationary semi-Riemannian metric of $M$
in the sense of Definition 3.1.3 p. 41 in \cite{k}.

On the other hand, a homomorphism
$P:\mathcal{X}\to \Omega^1(M)$ is
{\em stationary} if
$$
XP(Y)(Z)=P([X,Y])(Z)+P(Y)([X,Z]),
\,\,\,\,\,\,\,\,\,\,\forall (X,Y,Z)\in Ker(P)\times \mathcal{X}\times \mathcal{X}.
$$
This definition is also motivated by \cite{k}.

Stationary vector fields and homomorphisms
are related via the following lemma.

\begin{lemma}
 \label{kupeli}
$X\in \mathcal{X}$ is stationary with respect to $g$ if and only if
the homomorphism of $C^\infty(M)$-modules $P^X:\mathcal{X}\to \Omega^1(M)$ does, where
$$
P^X(Y)(Z)=g( Y,Z/X),
\,\,\,\,\,\,\,\,\forall Y,Z\in \mathcal{X}.
$$
\end{lemma}

\begin{proof}
By Lemma \ref{sl-dgw}-(1) we have
$X\in Ker(P^X)$ therefore $X$ is stationary if $P^X$
does.
Conversely, suppose that $X$ is stationary. To prove that $P^X$ also does it suffices to prove
$YP^X(Z)(Z)=2P^X([Y,Z])(Z)$, $\forall (Y,Z)\in Ker(P^X)\times \mathcal{X}$ or, equivalently,
$$
Yg( Z,Z/X\rangle=2g( [Y,Z],Z/X),
\,\,\,\,\,\,\,\forall (Y,Z)\in Ker(P^X)\times \mathcal{X}.
$$
With this in mind we
fix $(Y,Z)\in Ker(P^X)\times \mathcal{X}$.
Obviously we only have to check the above identity in $M_0$, the complement of the closure
of the interior of the set zeroes of $X$ and $Y$.
As $Y\in Ker(P^X)$ Definition \ref{2-inner}-(1) implies that
there is $f_0\in C^\infty(M_0)$ such that
$Y=f_0X$ in $M_0$.
But $X$ is stationary, so we have in $M_0$ that
\[
\begin{array}{ccl}
Yg( Z,Z/X)
& = &
f_0Xg( Z,Z/X)\\
& = &
2f_0g(  [X,Z],Z/X)\\
& = &
2 g( [f_0X,Z],Z/X)\\
& = &
2 g( [Y,Z],Z/X)
\end{array}
\]
and the proof follows.
\end{proof}

\begin{lemma}
\label{le1}
If a torsion free compatible $\Omega^1(M)$-pseudoconnection of $TM$ has a Koszul derivative, then
its principal homomorphism is stationary.
\end{lemma}

\begin{proof}
Let $\nabla $ be a torsion free compatible $\Omega^1(M)$-pseudoconnection with principal homomorphism
$P$. We have that $P$ is symmetric since $\nabla$ is compatible.
Therefore,
\[
\begin{array}{ccl}
XP(Y)(Z) & = &
\nabla_XY(Z)+\nabla_XZ(Y)\\
& = &
\nabla_YX(Z)+\nabla_ZX(Y)+P([X,Y])(Z)+P([X,Z])(Y)\\
& = &
Y(P(X)(Z))+Z(P(X)(Y))-\nabla_YZ(X)-\nabla_ZY(X)+\\
&  &
P([X,Y])(Z)+P([X,Z])(Y)
\end{array}
\]
proving
\begin{equation}
 \label{e1}
XP(Y)(Z)-P([X,Y])(Z)-P(Y)([X,Z])=YP(X)(Z)+ZP(X)(Y)-
\end{equation}
$$
\nabla_YZ(X)-\nabla_ZY(X),
$$
for all $X,Y,Z\in \mathcal{X}$.

Now suppose that $\nabla$ has a Koszul derivative $D$.
If $X\in Ker(P)$ then $P(X)=0$ and so
$$
\nabla_YZ(X)=P(D_YZ)(X)=P(X)(D_YZ)=0.
$$
It follows that $P$ is stationary by (\ref{e1}).
\end{proof}

Now we state an auxiliary lemma.

\begin{lemma}
\label{frakW}
For every $2$-Riemannian metric $g$ in $M$ the map
$\mathfrak{W}:\mathcal{X}\times \mathcal{X}\times\mathcal{X}\to C^\infty(M)$ defined by
$$
\mathfrak{W}(X,Y,Z)=g([Z,X],Y/Z)+g(X,[Z,Y]/Z)-Zg(X,Y/Z),
\,\,\,\,\forall X,Y,Z\in\mathcal{X}
$$
satisfies the following properties:
\begin{enumerate}
\item
$X$ is stationary for $g$ if and only if $\mathfrak{W}(Y,Y,X)=0$
for all $Y\in\mathcal{X}$.
\item
If $X,Y,Z\in \mathcal{X}$ and $\varphi\in C^\infty(M)$, then
$$
\mathfrak{W}(X,Y,\varphi Z)=\varphi^3\mathfrak{W}(X,Y,Z)-\varphi Z(\varphi^2)
g(X,Y/Z).
$$
\end{enumerate}
\end{lemma}

\begin{proof}
The first part follows directly from the definition of stationary vector field.
The second one is a direct computation,
\[
\begin{array}{ccl}
\mathfrak{W}(X,Y,\varphi Z)
& = &
g([\varphi Z,X],Y/\varphi Z)+
g(X,[\varphi Z,Y]/\varphi Z)-(\varphi Z)g(X,Y/\varphi Z)\\
& = &
\varphi^2g(-X(\varphi)Z+\varphi [Z,X],Y/Z)+\\
& &
\varphi^2g(X,-Y(\varphi)Z+\varphi[Z,Y]/Z)-
\varphi\cdot Z(\varphi^2g(X,Y/Z))\\
& = &
\varphi^3(g([Z,X],Y/Z)+g(X,[Z,Y]/Z))-\varphi^3\cdot Zg(X,Y/Z)-\\
& &
\varphi\cdot Z(\varphi^2)\cdot g(X,Y/Z)\\
& = &
\varphi^3\mathfrak{W}(X,Y,Z)-\varphi Z(\varphi^2)
g(X,Y/Z).
\end{array}
\]
\end{proof}

\begin{lemma}
 \label{le2}
For every non zero vector field $X\in \mathcal{X}$ the orbit $C^\infty(M)\cdot X$
of $X$ under the natural action $C^\infty(M)\times \mathcal{X}\to\mathcal{X}$
contains at least one non stationary vector field with respect to $g$.
\end{lemma}

\begin{proof}
Suppose that there is a non zero vector field $X$ such that $\varphi X$ is stationary
for all $\varphi\in C^\infty(M)$.
Then,
$$
\mathfrak{W}(Y,Y,\varphi X)=0,
\,\,\,\,\,\forall Y\in \mathcal{X}, \forall \varphi\in C^\infty(M),
$$
by Lemma \ref{frakW}-(1) and so
$$
\varphi X(\varphi^2)g( Y,Y/X)=0,
\,\,\,\,\,\,\,\,\,\forall Y\in \mathcal{X},\forall \varphi\in C^\infty(M)
$$
by Lemma \ref{frakW}-(2).
From this it follows that
$\varphi X(\varphi^2)$ vanishes outside the set of zeroes of $X$, for all $\varphi$.
But $\varphi X(\varphi^2)$ also vanishes in the set of zeroes of $X$.
Therefore
$$
\varphi X(\varphi^2)=0,
\,\,\,\,\,\,\,\,\forall \varphi\in C^\infty(M).
$$
From this we obtain $X=0$ which is absurd by hypothesis.
This contradiction proves the result.
\end{proof}

Now we can state the main result of this subsection.

\begin{thm}
\label{thBB}
The $2$-Riemannian pseudoconnections have no Koszul derivatives.
\end{thm}

\begin{proof}
Suppose that there is a $2$-Riemannian manifold $(M,g)$ whose $2$-Riemannian pseudoconnection
$\nabla=\nabla^{g}$ has a Koszul derivative $D$.
For all $X\in \mathcal{X}$ we define
$\nabla^X:\mathcal{X}\times\mathcal{X}\to \Omega^1(M)$ by
$$
\nabla^X_YZ(W)=\nabla_YZ(W,X),
\,\,\,\,\,\,\,\,\,\forall Y,Z\in \mathcal{X}.
$$
Then, $\nabla^X$ is a torsion free compatible $\Omega^1(M)$-pseudoconnection of $TM$ with principal
homomorphism $P^X$ as in Lemma \ref{kupeli}.
But
\[
\begin{array}{ccl}
P^X(D_YZ)(W)
& = &
g( D_YZ,W/X)\\
& = &
P(D_YZ)(W,X)\\
& = &
\nabla_YZ(W,X)\\
& = &
\nabla_Y^XZ(W)
\end{array}
\]
where $P=P^{g}$ above is the principal homomorphism of $\nabla$.
So $D$ is a Koszul derivative of $\nabla^X$. Therefore
$P^X$ is stationary by Lemma \ref{le1}.
Consequently every $X\in \mathcal{X}$ is stationary by Lemma \ref{kupeli} which contradicts
Lemma \ref{le2}.
This finishes the proof.
\end{proof}

\subsection{Non-existence of Koszul derivatives II}
In this subsection
we obtain another proof of Theorem \ref{thBB} based on the following observation:
The target module $\mathfrak{D}$ in the definition of $\mathfrak{D}$-pseudoconnection $\nabla$
on vector bundles $\xi$ over a manifold $M$ is not necessarily unique.
In fact, it may be replaced by a submodule of $\mathfrak{D}$ containing both $Im(\nabla)$ and $Im(P)$, where $P$ is the principal homomorphism of $\nabla$.
The best possible is $\mathfrak{D}(\nabla)$, the submodule of $\mathfrak{D}$ generated by $Im(\nabla)\cup Im(P)$.
There are examples where $\mathfrak{D}(\nabla)=Im(P)$
as in the case of the unique torsion free compatible $\Omega^1(M)$-pseudoconnection
of $TM$ with principal homomorphism
$P(X)(Y)=h(X,Y)$ induced by a Riemannian metric $h$ of $M$.
Therefore, in such a case we have the inclusion
$$
Im(\nabla)\subset Im(P).
$$
The situation for $2$-Riemannian pseudoconnections
will be completely different as
we shall prove in that case that
\begin{equation}
\label{identity}
Im(\nabla)\cap Im(P)=\{0\}.
\end{equation}
To prove it we shall need some previous lemmas.
Given a vector field $X$ we denote by $Sing(X)$ the set of zeroes of $X$.

Observe that if $X,Y\in \mathcal{X}$ satisfy
$Sing(X)\cup Sing(Y)=M$ then $[X,Y]=0$.
Indeed, it follows from the definition of $[X,Y](p)$,
$$
[X,Y](p)(\varphi)=X(p)(Y(\varphi))-Y(p)(X(\varphi)),
\,\,\,\,\forall \varphi\in C^\infty(M),
$$
that $[X,Y](p)=0$ if $p\in Sing(X)\cap Sing(Y)$.
Now suppose that $p\in M\setminus (Sing(X)\cap Sin(Y))$.
We can assume without loss of generality that $p\not\in Sing(X)$.
As $Sing(X)\cup Sing(Y)=M$ we have that $Y$
vanishes not only in $p$ but also in a neighborhood of $p$.
Hence $Y(\varphi)$ vanishes in such a neighborhood, for all $\varphi\in C^\infty(M)$, therefore $[X,Y](p)=0$
for all $p\in M\setminus(Sing(X)\cap Sing(Y))$
thus $[X,Y]=0$.

More consequences of the identity $Sing(X)\cup Sing(Y)=M$ are given below.

\begin{lemma}
\label{auxiliar0}
Let $(M,g)$ be a $2$-Riemannian manifold.
If $X,Y\in\mathcal{X}$ satisfy
$Sing(X)\cup Sing(Y)=M$, then
$$
g(X,Y/Z)=Xg(Y,Z/W)=g([Z,X],Y/W)=0,
\,\,\forall Z,W\in\mathcal{X}.
$$
\end{lemma}

\begin{proof}
It follows from the definition that $g(X,Y/Z)=0$ for all $Z\in \mathcal{X}$ whenever $Sing(X)\cup Sing(Y)=M$.
On the other hand, $Xg(Y,Z/W)$
clearly vanishes at $Sing(X)$ and, since $Y$ vanishes at $M\setminus Sing(X)$
which is open, we obtain that $Xg(Y,Z/W)$ vanishes in $M\setminus Sing(X)$
as well. Hence $Xg(Y,Z/W)=0$ in $M\setminus Sing(X)$ therefore
$Xg(Y,Z/W)=0$.

Finally we consider $g([Z,X],Y/W)$
which clearly vanishes at $Sing(Y)$.
As $Sing(X)\cup Sing(Y)=M$ we have that $[Z,X]$ vanishes
at $M\setminus Sing(Y)$.
Hence $g([Z,X],Y/W)$ vanishes at $M\setminus Sing(Y)$ as well, so
$g([Z,X],Y/W)=0$.
\end{proof}

\begin{clly}
\label{auxiliar2}
If
$X,Y\in\mathcal{X}$ satisfy $Sing(X)\cup Sing(Y)=M$,
then $\nabla_XY=0$
for all $2$-Riemannian pseudoconnection $\nabla$ of $M$.
\end{clly}

\begin{proof}
It follows from (\ref{pseudoconnection}) that
if $g$ is the $2$-Riemannian metric associated to $\nabla$, i.e.,
$\nabla=\nabla^g$ then
$$
\nabla_XY(Z,W)=\frac{1}{2}\{Xg(Y,Z/W)+Yg(Z,X/W)-Zg(X,Y/W)+
$$
$$
g([X,Y],Z/W)
+g([Z,X],Y/W)-g([Y,Z],X/W)\},
\,\,\,\,\,\forall X,Y,Z,W\in\mathcal{X}.
$$
Then, the result follows for
the six summands in the right-hand side of the
expression above vanish by Lemma \ref{auxiliar0} whenever $Sing(X)\cup Sing(Y)=M$.
\end{proof}

For the next lemma recall the auxiliary map $\mathfrak{W}$ in Lemma \ref{frakW}.

\begin{lemma}
\label{auxiliar1}
If $(M,g)$ is a $2$-Riemannian manifold, then the following properties
are equivalent for all $X,Y\in\mathcal{X}$:
\begin{description}
\item[(L1)]
$Sing(X)\cup Sing(Y)=M$.
\item[(L2)]
$g(X,Y/Z)=0$ for all $Z\in \mathcal{X}$.
\item[(L3)]
$\mathfrak{W}(X,Y,Z)=0$ for all $Z\in \mathcal{X}$.
\end{description}
\end{lemma}

\begin{proof}
We have that (L1) implies (L2) by Lemma \ref{auxiliar0}.

Let us prove that (L2) implies (L1).
Suppose that $X,Y\in\mathcal{X}$ satisfies (L2).
We know by Lemma \ref{sl-dgw}-(1) that $g(X,X/Y)=-g(X,Y/X+Y)$ so
$$
g(X,X/Y)=0
$$
by taking $Z=X+Y$ in (L2).
It then follows from the definition of $2$-inner product that
$X(p)$ and $Y(p)$ are linearly dependent for all $p\in M$.
Now fix $p\in M\setminus Sing(Y)$.
Since $X(p)$ and $Y(p)$ are linearly dependent
there are $a,b\in \mathbb{R}$ with
$a\neq 0$ or $b\neq 0$ so that
$$
aX(p)+bY(p)=0.
$$
Let us prove that $p\in Sing(X)$. Suppose by contradiction that $p\notin Sing(X)$.
Then, $X(p)\neq 0$ and so $b\neq 0$ for, otherwise, $b=0$ hence
$aX(p)=0$ yielding $a=0$ (because $X(p)\neq 0$) thus $a=b=0$ which contradicts the fact that $a\neq 0$ or $b\neq 0$.
Therefore, we can write
$$
Y(p)=\lambda X(p),
$$
where $\lambda=-\frac{a}{b}\in \mathbb{R}\setminus \{0\}$.
Replacing in (L2) evaluated at $p$ we get
$$
\lambda g(X(p),X(p)/Z(p))=g(X(p),Y(p)/Z(p))=0,
$$
for all $Z(p)\in T_pM$.
This necessarily implies $X(p)=0$ which contradicts $X(p)\neq 0$.
This contradiction shows that $p\in Sing(X)$.
Since $p\in Sing(X)$ we have that (L1) holds.

We have that (L1) implies (L3) for if
(L1) holds, then
$Zg(X,Y/Z)=g([Z,X],Y/Z)=g(X,[Z,Y]/Z)=0$ by
Lemma \ref{auxiliar0}.

To prove that (L3) implies (L2) we see from Lemma \ref{frakW} that if $\mathfrak{W}(X,Y,Z)=0$ for all $Z\in \mathcal{X}$ then
$$
\varphi Z(\varphi^2) g(X,Y/Z)=0,
\,\,\,\,\,\,\,\forall Z\in \mathcal{X}, \forall \varphi\in C^\infty(M)
$$
which is equivalent to (L2).
This proves the lemma.
\end{proof}

Now we can state the main result of this subsection.

\begin{thm}
\label{new-proof}
The identity (\ref{identity}) holds for every
$2$-Riemannian pseudoconnection $\nabla$.
\end{thm}

\begin{proof}
Let $g$ be the $2$-Riemannian metric associated to $\nabla$.
Then, $P=P^g$ where
$$
P^g(X)(Y,Z)=g(X,Y/Z),
\,\,\,\,\,\,\,\,\forall X,Y,Z\in\mathcal{X}.
$$
Fix $d\in Im(P)\cap Im(\nabla)$.
Then, there are $A\in \mathcal{X}$ and $X,Y\in \mathcal{X}$
such that
$$
d=P^g(A) \,\,\,\,\,\mbox{ and }
\,\,\,\,\,\,
d=\nabla_{X}Y.
$$
In particular,
$$
d(Z,W)=\frac{1}{2}\{Xg(Y,Z/W)+Yg(Z,X/W)-Zg(X,Y/W)+
$$
$$
g([X,Y],Z/W)
+g([Z,X],Y/W)-g([Y,Z],X/W\},
\,\,\,\,\,\forall Z,W\in\mathcal{X}.
$$
Taking $Z=W$ above and observing that $d(Z,Z)=P^g(A)(Z,Z)=g(A,Z/Z)=0$,
$\forall Z\in \mathcal{X}$, we get
$$
0=g([Z,X],Y/Z)+g(X,[Z,Y]/Z)
-Zg(X,Y/Z),
\,\,\,\,\,\,\forall Z\in \mathcal{X}.
$$
So $X,Y$ satisfies (L3) in Lemma \ref{auxiliar1}.
Then, $X,Y$ also satisfy (L1) in Lemma \ref{auxiliar1} and so
$\nabla_XY=0$ by Corollary \ref{auxiliar2}.
As $d=\nabla_XY$ we conclude that $d=0$ hence $Im(\nabla)\cap Im(P)=\{0\}$.
\end{proof}

{\flushleft{\bf Second proof of Theorem \ref{thBB}: }}
If a pseudoconnection $\nabla$ has a Koszul derivative, then
it would satisfy $Im(\nabla)\subset Im(P)$ where $P$ is its principal homomorphism.
But if $\nabla$ were $2$-Riemannian then we would have $Im(\nabla)\cap Im(P)=\{0\}$ by Theorem \ref{new-proof}. Hence in such a case we would have
$Im(\nabla)=Im(\nabla)\cap Im(P)=\{0\}$
which implies $\nabla=0$, a contradiction.
Thus $\nabla$ has no Koszul derivatives.
\qed

It seems that for every
$2$-Riemannian pseudoconnection $\nabla$ there is a direct sum
$\mathfrak{D}(\nabla)=Im(P)\oplus \mathfrak{I}(\nabla)$,
where $\mathfrak{I}(\nabla)$ is the submodule of $\mathfrak{D}$ generated by $Im(\nabla)$.

\subsection{Adapted ordinary pseudoconnections}
We have seem at the beginning of Section \ref{$2$-Riem-pseudo} how to associate a
$\mathfrak{D}^2_0$-pseudoconnection to any $2$-Riemannian manifold $(M,g)$.
It would be better however to associate an ordinary pseudoconnection instead.
A possible problem for such an association is
to give a better definition of compatibility between
$2$-Riemannian metrics and ordinary pseudoconnections.
The reasonable approach to obtain such a kind of compatibility is
to compute derivatives like $Xg(Y,Z/W)$ for arbitrary vector fields $X,Y,Z,W\in\mathcal{X}$.

As a motivation we shall do it in the case when
$g$ is a simple $2$-Riemannian metric generated by a Riemannian metric $h$ of $M$.
First of all observe that
\[
\begin{array}{ccl}
Xg(Y,Y/Z)
& = &
X(h(Y,Y)h(Z,Z)-h^2(Y,Z))\\
& = &
2h(\theta_XY,Y)h(Z,Z)+2h(\theta_XZ,Z)h(Y,Y)-\\
& &
2h(Y,Z)(h(\theta_XY,Z)+h(\theta_XZ,Y)),
\end{array}
\]
where $\theta$ is the Riemannian connection of $h$.
Hence
\[
\begin{array}{ccl}
Xg(Y,Y/Z)
& = &
2(h(\theta_XY,Y)h(Z,Z)-h(Y,Z)h(\theta_XY,Z))+\\
& &
2(h(\theta_XZ,Z)h(Y,Y)-h(Y,Z)h(\theta_XZ,Y))\\
& = &
2g(\theta_XY,Y/Z)+2g(\theta_XZ,Z/Y).
\end{array}
\]
Using it we obtain both
\[
\begin{array}{ccl}
Xg((Y+Z,Y+Z/W)
& = &
2g(\theta_X(Y+Z),Y+Z/W)+2g(\theta_XW,W/Y+Z)\\
& = &
2g(\theta_XY,Y/W)+2g(\theta_XY,Z/W)+2g(\theta_XZ,Y/W)+\\
& &
2g(\theta_XZ,Z/W)+
2g(\theta_XW,W/Y+Z)
\end{array}
\]
and
\[
\begin{array}{ccl}
Xg(Y+Z,Y+Z,/W)
& = &
X(g(Y,Y/W)+g(Z,Z/W)+2g(Y,Z/W))\\
& = &
2g(\theta_XY,Y/W)+2g(\theta_XW,W/Y)+2g(\theta_XZ,Z/W)\\
& &
2g(\theta_XW,W/Z)+2Xg(Y,Z/W),
\end{array}
\]
for all $X,Y,Z,W\in \mathcal{X}$.
Therefore, $\theta$ satisfies the identity
\begin{equation}
\label{adapted}
Xg(Y,Z/W)=g(\theta_XY,Z/W)+g(Y,\theta_XZ/W)+g(\theta_XW,W/Y+Z)-
\end{equation}
$$
\quad \quad \quad \quad \quad \quad \quad \quad g(\theta_XW,W/Y)-g(\theta_XW,W/Z),
\,\,\,\,\,\,\forall X,Y,Z,W\in\mathcal{X}.
$$

Let us use this identity as definition.

\begin{defi}
\label{defi-adapted}
An ordinary pseudoconnection $\theta$ of $M$ is {\em adapted}
to $g$ if
(\ref{adapted}) holds for all $X,Y,Z,W\in \mathcal{X}$.
\end{defi}

We have then proved the following.

\begin{prop}
\label{adapted-prop}
Let $g$ be a simple $2$-Riemannian metric of $M$ generated by a Riemannian metric $h$. Then, the Riemannian connection of $h$ is an adapted
torsion-free ordinary pseudoconnection of $g$.
\end{prop}

\begin{lemma}
\label{adapted-lemma}
An ordinary pseudoconnection $\theta$ of $M$
is adapted to $g$ if and only if $Xg(Y,Y/Z)=
2g(\theta_XY,Y/Z)+2g(\theta_XZ,Z/Y)$, $\forall X,Y,Z\in \mathcal{X}$.
\end{lemma}

In view of Proposition \ref{adapted-prop} it would be interesting
to investigate existence and uniqueness of adapted ordinary pseudoconnections for a given $2$-Riemannian metric $g$.
In this direction we only have the following short result.
Recall that a $2$-Riemannian metric $\overline{g}$ is said to be
{\em conformally equivalent} to another $2$-Riemannian metric $g$ if
$\overline{g}=\lambda\cdot g$ for some positive map $\lambda\in C^\infty(M)$.

\begin{prop}
\label{adapted-lemma'}
Let $g$ be a $2$-Riemannian metric with an adapted ordinary pseudoconnection $\theta$.
If $\overline{g}=\lambda \cdot g$ is a $2$-Riemannian metric
conformally equivalent to $g$,
then the map $\overline{\theta}:\mathcal{X}\times\mathcal{X}\to\mathcal{X}$
define by
$$
\overline{\theta}_XY=\theta_XY+\frac{X(\lambda)}{4\lambda},
\,\,\,\,\,\,\,\forall X,Y\in\mathcal{X},
$$
is an adapted ordinary pseudoconnection of $\overline{g}$.
\end{prop}

\begin{proof}
A direct computation shows that $\overline{\theta}$ is
an ordinary pseudoconnection of $M$ with the same principal homomorphism of $\theta$.
On the other hand, for all $X,Y,Z\in\mathcal{X}$ one has
\[
\begin{array}{ccl}
X\overline{g}(Y,Y/Z)
& = &
X(\lambda)g(Y,Y/Z)+\lambda Xg(Y,Y/Z)\\
&  &
  \\
& = &
X(\lambda)g(Y,Y/Z)+
\lambda(2g(\theta_XY,Y/Z)+2g(\theta_XZ,Z/Y))\\
&  &
  \\
& = &
\frac{X(\lambda)}{\lambda}\overline{g}(Y,Y/Z)+2\overline{g}(\theta_XY,Y/Z)+\overline{g}(\theta_XZ,Z/Y)\\
&  &
  \\
& = &
2\overline{g}(\frac{X(\lambda)}{4\lambda}Y+\theta_XY,Y/Z)+2\overline{g}
(\frac{X(\lambda)}{4\lambda}Z+\theta_XZ,Z/Y)\\
&  &
  \\
& = &
2\overline{g}(\overline{\theta}_XY,Y/Z)+2\overline{g}(\overline{\theta}_XZ,Z/Y).
\end{array}
\]
Therefore, $\overline{\theta}$ is adapted by Lemma \ref{adapted-lemma}.
\end{proof}

The above proposition cannot be used to construct non-simple
$2$-Riemmanian metrics with adapted ordinary pseudoconnections.
This is because $2$-Riemannian metrics conformally equivalent to
simple ones are simple too.
However it can be used to prove the existence of
$2$-Riemannian metrics in $\mathbb{R}^2$ exhibiting two different
adapted ordinary pseudoconnections.

Indeed, consider a positive map $f\in C^\infty(\mathbb{R}^2)$ whose differential
$df:\mathcal{X}\times\mathcal{X}\to\mathcal{X}$ as defined in (\ref{df}) is
non-zero.
Define the $2$-Riemannian metric $g=f\cdot g^{st}$ in $\mathbb{R}^2$ and
$\overline{\theta}:\mathcal{X}\times\mathcal{X}\to\mathcal{X}$ by
$$
\overline{\theta}_XY=\theta^{st}_XY+\frac{X(f)}{4f},
\,\,\,\,\,\,\,\,\,\forall X,Y\in \mathcal{X},
$$
where $\theta^{st}$ is the Riemannian connection of the standard Euclidean
product of $\mathbb{R}^2$.
It follows from Lemma \ref{adapted-lemma'} that $\overline{\theta}$
is an adapted ordinary pseudoconnection of $g$.

On the other hand, $g$ is simple since $g^{st}$ is.
Hence, by Proposition \ref{adapted-prop}, the Riemannian connection $\hat{\theta}$ of the Riemannian metric
generating $g$ is also an adapted ordinary pseudoconnection of $\overline{g}$.
But $\hat{\theta}$ is torsion free whereas
$\overline{\theta}$ is not because $
\overline{\theta}_XY-\overline{\theta}_YX=
[X,Y]+\frac{df(X,Y)}{4\lambda}$, for all $X,Y\in\mathcal{X}$.
Therefore $\overline{\theta}\neq\hat{\theta}$ and the result follows.
Despite it seems possible to prove the uniqueness of
adapted ordinary pseudoconnections but among the torsion free ones.

The next result explains how
adapted ordinary pseudoconnections can be used to compute
$2$-Riemannian pseudoconnections.
Its proof is a direct computation which is left to the reader.

\begin{prop}
\label{computa}
If $g$ is a $2$-Riemmanian metric with an adapted
torsion free ordinary connection $\theta$ then
the $2$-Riemannian pseudoconnection $\nabla^g$ of $g$ splits as
$\nabla^g=g*\theta+\Omega^{g,\theta}$,
where $(g*\theta)_XY(Z,W)=g(\theta_XY,Z/W)$
and
$$
\Omega^{g,\theta}_XY(Z,W)=
\frac{1}{2}\{
g(\theta_{Z-X}W,W/Y)-g(\theta_{X+Y}W,W/Z)+
g(\theta_{Z-Y}W,W/X)+
$$
$$
g(\theta_XW,W/Y+Z)+
g(\theta_YW,W/X+Z)-g(\theta_ZW,W/X+Y)\},
$$
for all $X,Y,Z,W\in \mathcal{X}$.
\end{prop}

Let us use this proposition to compute the $2$-Riemannian pseudoconnection $\nabla^{g^{st}}$ of the standard $2$-Riemmanian metric $g^{st}$ in $\mathbb{R}^2$. In such a case Lemma \ref{sl-dgw}-(2) implies
$
g^{st}(X,Y/ Z) =\det\left(\begin{matrix}
X_1&X_2\\
Z_1&Z_2
\end{matrix}
\right)\cdot
\det\left(\begin{matrix}
Y_1&Y_2\\
Z_1&Z_2
\end{matrix}
\right),
$
where $K=(K_1,K_2)$ are the coordinates
of $K\in \{X,Y,Z\}$.
Since the Riemannian connection
of the standard metric in $\mathbb{R}^2$ is given
by $\theta_XY=(X(Y_1),X(Y_2))$
we get from Proposition \ref{computa} the following formula
for $\nabla^{g^{st}}$:

\[
\begin{array}{ccl}
\nabla^{g^{st}}_XY(Z,W) & = &
\det
\left(
\begin{matrix}
Z_1&Z_2\\
W_1&W_2
\end{matrix}
\right)
\cdot
\det
\left(
\begin{matrix}
X(Y_1)&X(Y_2)\\
W_1&W_2
\end{matrix}
\right)+\\
& + & 
\frac{1}{2}
\cdot
\{
\det
\left(
\begin{matrix}
W_1&W_2\\
Y_1+Z_1&Y_2+Z_2
\end{matrix}
\right)
\cdot
\det
\left(
\begin{matrix}
X(W_1)&X(W_2)\\
Y_1+Z_1&Y_2+Z_2
\end{matrix}
\right)+\\
& + &
\det
\left(
\begin{matrix}
W_1&W_2\\
X_1+Z_1&X_2+Z_2
\end{matrix}
\right)
\cdot
\det
\left(
\begin{matrix}
Y(W_1)&Y(W_2)\\
X_1+Z_1&X_2+Z_2
\end{matrix}
\right)+\\
& + &
\det
\left(
\begin{matrix}
W_1&W_2\\
X_1+Y_1&X_2+Y_2
\end{matrix}
\right)
\cdot
\det
\left(
\begin{matrix}
Z(W_1)&Z(W_2)\\
X_1+Y_1&X_2+Y_2
\end{matrix}
\right)+\\
& + &
\det
\left(
\begin{matrix}
W_1&W_2\\
Y_1&Y_2
\end{matrix}
\right)
\cdot
\det
\left(
\begin{matrix}
(Z-X)(W_1)&(Z-X)(W_2)\\
Y_1&Y_2
\end{matrix}
\right)-\\
& - &
\det
\left(
\begin{matrix}
W_1&W_2\\
Z_1&Z_2
\end{matrix}
\right)
\cdot
\det
\left(
\begin{matrix}
(X+Y)(W_1)&(X+Y)(W_2)\\
Z_1&Z_2
\end{matrix}
\right)+\\
& + &
\det
\left(
\begin{matrix}
W_1&W_2\\
X_1&X_2
\end{matrix}
\right)
\cdot
\det
\left(
\begin{matrix}
(Z-Y)(W_1)&(Z-Y)(W_2)\\
X_1&X_2
\end{matrix}
\right)
\}.
\end{array}
\]

\subsection{Stationary vector fields}
As we have seen, the stationary vector fields played an important role
in the first proof of Theorem \ref{thBB}.
This motivates the question whether such vector
fields exist for a given $2$-Riemannian metric.
Here we consider the case of $2$-Riemannian metrics $g$ on open subsets of
$\mathbb{R}^2$ conformally equivalent to $g^{st}$.
In such a case we prove that the stationary vector fields $X$ are precisely the solutions of the differential equation
\begin{equation}
\label{div}
2div(X)+X(\ln(\lambda))=0,
\end{equation}
where $div(X)$ above is the divergence of $X$.
The proof is based on the following lemma which gives
the property of $g^{st}$ we shall need.

\begin{lemma}
\label{s2}
If $X,Y\in \mathcal{X}(\mathbb{R}^2)$ then
$$
Xg^{st}(Y,Y/X)-2\cdot g^{st}([X,Y],Y/X)=2\cdot div(X)\cdot g^{st}(Y,Y/X).
$$
\end{lemma}

\begin{proof}
Given $X,Y\in\mathcal{X}(\mathbb{R}^2)$ there are $C^\infty$ maps $a,b,c,d$ such that
$$
X=a\frac{\partial}{\partial x}+b\frac{\partial}{\partial y}
\,\,\,\,\,\,\mbox{ and }
\,\,\,\,\,\,Y=c\frac{\partial}{\partial x}+d\frac{\partial}{\partial y},
$$
where $\frac{\partial}{\partial x},\frac{\partial}{\partial y}$ is the standard basis
of $\mathcal{X}(\mathbb{R}^2)$.
By the definition of $g^{st}$ we have
$$
g^{st}\left( \frac{\partial}{\partial x},\frac{\partial}{\partial x}\bigg/\frac{\partial}{\partial y}
\right)=1
$$
and so
\begin{equation}
 \label{equs1}
g^{st}( Y,Y/X)=
\left[\det
\left(
\begin{matrix}
a & b\\
c & d
\end{matrix}
\right)
\right]^2
\end{equation}
by Lemma \ref{sl-dgw}-(2).
On the other hand, if $\theta$ is the Riemannian connection of the Euclidean metric in $\mathbb{R}^2$ then
$$
\theta_XX=X(a)\frac{\partial}{\partial x}+ X(b)\frac{\partial}{\partial y}
\,\,\,\,\mbox{ and }
\,\,\,\,\theta_YX=Y(a)\frac{\partial}{\partial x}+ Y(b)\frac{\partial}{\partial y}.
$$
See \cite{dcm}.
Therefore, applying Lemma \ref{sl-dgw}-(2) twice we get
$$
g^{st}(\theta_XX,X/Y)+g^{st}(\theta_YX,Y/X)=
\left[
\det
\left(
\begin{matrix}
X(a) & c\\
X(b) & d
\end{matrix}
\right)
-
\det
\left(
\begin{matrix}
Y(a) & a\\
Y(b) & b
\end{matrix}
\right)
\right]
\cdot
$$
$$
\det
\left(
\begin{matrix}
a & b\\
c & d
\end{matrix}
\right).
$$
But straightforward computations yield
$$
\det
\left(
\begin{matrix}
X(a) & c\\
X(b) & d
\end{matrix}
\right)=
ad\frac{\partial a}{\partial x}+bd\frac{\partial a}{\partial y}-ac\frac{\partial b}{\partial x}-
bc\frac{\partial b}{\partial y}
$$
and
$$
\det
\left(
\begin{matrix}
Y(a) & a\\
Y(b) & b
\end{matrix}
\right)=
bc\frac{\partial a}{\partial x}+bd\frac{\partial a}{\partial y}-ac\frac{\partial b}{\partial x}-
ad\frac{\partial b}{\partial y}.
$$
which together with (\ref{equs1}) yield
$$
div(X)\cdot g^{st}(Y,Y/X)=g^{st}(\theta_XX,X/Y)+g^{st}( \theta_YX,Y/X)
$$
since $div(X)=\frac{\partial a}{\partial x}+\frac{\partial b}{\partial y}$ is the divergence of $X$.
Therefore,
\[
\begin{array}{ccl}
div(X)\cdot g^{st}(Y,Y/X)
& = &
g^{st}(\theta_XX,X/Y)+g^{st}(\theta_YX,Y/X)\\
& = &
g^{st}(\theta_XX,X/Y)+g^{st}(\theta_XY,Y/X)-\\
& &
g^{st}([X,Y],Y/X).
\end{array}
\]
On the other hand, $\theta$ is an adapted ordinary pseudoconnection of $g^{st}$ by
Proposition \ref{adapted-prop}.
Hence
$$
g^{st}(\theta_XX,X/Y)+g^{st}(\theta_XY,Y/X)
=\frac{Xg^{st}(Y,Y/X)}{2}.
$$
Replacing above we get
$$
div(X)\cdot g^{st}(Y,Y/X)=
\frac{Xg^{st}(Y,Y/X)}{2}-
g^{st}([X,Y],Y/X)
$$
which implies the result.
\end{proof}

Now we can state the main result of this subsection.

\begin{thm}
\label{thCC}
Let $g=\lambda\cdot g^{st}$ be a $2$-Riemannian metric conformally equivalent to $g^{st}$. Then, the stationary vector fields with respect to
$g$ are precisely the solutions $X$ of (\ref{div}).
\end{thm}

\begin{proof}
We know from Remark \ref{rk2} that $X\in\mathcal{X}$ is stationary with respect to a $2$-Riemannian metric $g$ if and
only if
$$
Xg(Y,Y/X)=2g([X,Y],Y/X),
\,\,\,\,\,\,\forall Y\in \mathcal{X}.
$$
In the specific case $g=\lambda\cdot g^{st}$ this last expression is equivalent to
$$
\frac{X(\lambda)}{\lambda}g^{st}(Y,Y/X)+Xg^{st}(Y,Y/X)=2g^{st}([X,Y],Y/X),
\,\,\,\,\,\forall Y\in \mathcal{X}.
$$
Applying Lemma \ref{s2} we get the equivalent equality
$$
(2div(X)+X(\ln\lambda))g^{st}(Y,Y/X)=0,
\,\,\,\,\,\,\forall Y\in \mathcal{X}.
$$
Obviously if $X$ is a solution of (\ref{div}) then $X$ is stationary by the above equality.
Conversely if $X$ is stationary, then $X$ satisfies the above equality
hence
$2div(X)+X(\ln\lambda)=0$ outside the set of zeroes of $X$.
But $2div(X)+X(\ln\lambda)=0$ in the interior of the set of zeroes of $X$ too.
Therefore, $2div(X)+X(\ln\lambda)=0$ everywhere hence $X$ solves (\ref{div}).
This proves the result.
\end{proof}

\vspace{10pt}

Applying this theorem to the constant map $\lambda=1$
we immediately obtain the following corollary.

\begin{clly}
 \label{c1}
The stationary vector fields with respect to $g^{st}$ 
in $\mathbb{R}^2$ are the divergence free ones.
\end{clly}

One more consequence is the existence of stationary vector
fields for certain $2$-Riemannian metrics in $\mathbb{R}^2$.

\begin{clly}
Every $2$-Riemannian metric of $\mathbb{R}^2$ conformally equivalent
to $g^{st}$ has stationary vector fields.
\end{clly}

\begin{proof}
This follows from the fact that (\ref{div}) has a solution for all
$\lambda$.
\end{proof}

\end{document}